\documentclass[11pt]{article}
\usepackage{amsmath}
\usepackage{amsfonts}
\usepackage{amsthm}
\usepackage{amssymb, setspace}
\usepackage{mathtools}
\usepackage{lipsum}
\usepackage{color,graphicx,epsfig,geometry,fancyhdr,hyperref}
\usepackage[T1]{fontenc}
\usepackage{float}
\usepackage{subfig}
\usepackage{commath}
\usepackage{bbm}
\usepackage{mathrsfs,fleqn}
\usepackage{pgfplots}
\pgfplotsset{compat=newest}
\newtheorem{theorem}{Theorem}[section]

\newtheorem{lemma}{Lemma}[section]

\newcommand{\N}{\mathbb{N}}

\newcommand{\R}{\mathbb{R}}
\newcommand{\C}{\mathbb{C}}

\usepackage{ulem}

\graphicspath{{paper-pics/}}

\begin{document}

\begin{flushleft}
\Large 
\noindent{\bf \Large The anisotropic interior transmission eigenvalue problem with a conductive boundary}
\end{flushleft}

\vspace{0.2in}
{\bf  \large Victor Hughes and Isaac Harris}\\
\indent {\small Department of Mathematics, Purdue University, West Lafayette, IN 47907 }\\
\indent {\small Email: \texttt{vhughes@purdue.edu} and \texttt{harri814@purdue.edu} }\\

{\bf  \large Jiguang Sun}\\
\indent {\small Department of Mathematics, Michigan Technological University, Houghton, MI 49931 }\\
\indent {\small Email: \texttt{jiguangs@mtu.edu} }\\

\begin{abstract}
In this paper, we study the transmission eigenvalue problem for an anisotropic material with a conductive boundary. We prove that the transmission eigenvalues for this problem exist and are at most a discrete set. We also study the dependence of the transmission eigenvalues on the physical parameters and prove that the first transmission eigenvalue is monotonic. We then consider the limiting behavior of the transmission eigenvalues as the conductive boundary parameter $\eta$ vanishes or goes to infinity in magnitude. Finally, we provide some numerical examples on three different domains to demonstrate our theoretical results. 
\end{abstract}



\section{Introduction}\label{intro}
In this paper, we study the transmission eigenvalue problem for an anisotropic material with a conductive boundary. This can be seen as covering the boundary of the scatterer with a thin highly conductive layer. This is modeled by assuming that there is a jump across the boundary of the scatterer with respect to the `normal' derivatives of the total field. Transmission eigenvalues are a relevant area of research due to the fact that they can be used to retrieve information about the material properties of the scattering object (see e.g. \cite{cavities,cakoni2014homogenization,GP,te-cbc3}). It is well known that the transmission eigenvalues can be recovered by the scattering data \cite{far field data,te-cbc2,armin}. Therefore, these eigenvalues can be used as a target signature, i.e., one can determine defects in a material from the measured scattering data. The relevant inverse problem is, given the eigenvalues, determine/estimate parameters in the differential operator. We analyze the dependence of the transmission eigenvalues on the conductivity parameter. 

As we shall see, the eigenvalue problem is non-self-adjoint and non-linear. The standard theory for eigenvalues for an elliptic operator does not apply. This makes the study of the transmission eigenvalues interesting and challenging mathematically. We refer to \cite{corner-nonscatt,aniso-nonscatt,salo-nonscatt} on some recent work connecting the transmission eigenvalues to non-scattering frequencies as well as \cite{fem-te,JiSun2013,mfs-te,mfs-anisote,eig-FEM-book} for some numerical algorithms to compute the eigenvalues. Here we study the discreteness, existence, and dependence on the parameters for real transmission eigenvalues.

Now we introduce the scattering problem associated with the transmission eigenvalues of interests. To this end, let $D \subset \mathbb{R}^d$ for $d=$ 2 or 3 be a simply connected open set with Lipschitz boundary $\partial D$ and $\nu$ be the unit outward normal vector. The region $D$ denotes the scatterer that we illuminate with an incident plane wave $u^i:=\text{e} ^{\text{i} kx\cdot\hat{y}}$ such that $k>0$ is the wave number and $\hat{y} \in \mathbb{S}^{d-1}$ denotes the incident direction. Let $A(x) \in L^\infty (D, \mathbb{R}^{d\times d})$ be a symmetric matrix valued function that is uniformly positive definite in $D$ satisfying 
\begin{equation*}
0< A_{\text{min}}=\inf\limits_{x\in D} \inf\limits_{|\xi|=1} \overline{\xi} \cdot A(x)\xi  \quad \text{and}\quad A_{\text{max}}=\sup\limits_{x\in D} \sup\limits_{|\xi|=1} \overline{\xi} \cdot A(x)\xi .
\end{equation*}
We also assume that the refractive index $n(x) \in L^\infty (D)$ such that
$$0< n_{\text{min}} \leq n(x) \leq n_{\text{max}} \quad \text{ a.e. in } \,\, D.$$
Lastly, let the conductivity parameter $\eta \in L^\infty(\partial D)$ satisfy
$$ \eta_{\text{min}} \leq \eta(x) \leq \eta_{\text{max}} \quad \text{ a.e. on } \,\, \partial D.$$

The direct scattering problem for an anisotropic material with a conductive boundary condition is formulated as follows: find $u^s \in H^1_{loc}(\mathbb{R}^d)$ such that
\begin{align}
 \Delta u^s + k^2u^s = 0 \quad \text{in $\mathbb{R}^d \setminus \overline{D}$} \quad \text{ and } \quad  \nabla \cdot A\nabla u+k^2nu=0 \quad &\text{in $D$}, \label{direct:1} \\ 
(u^s+u^i)^+ = u^-  \quad \text{ and } \quad  \partial_\nu (u^s+u^i)^+ = \nu \cdot A\nabla u^- - \eta u \quad &\text{on $\partial D$}, \label{direct:2} 
\end{align}
where $u = u^s+u^i$ denotes the total field. Here the superscripts $+$ and $-$ demonstrate approaching the boundary from the outside and inside of $D$, respectively. The scattered field satisfies the radiation condition 
$$\partial_r u^s -\text{i} ku^s = \mathcal{O}\Big(\frac{1}{r^{(d+1)/2}}\Big)  \quad \text{ as } \quad r \to \infty$$
uniformly with respect to $\hat{x}=\frac{x}{r}$ where $r=|x|$. Similar arguments as in \cite{fmconductbc} lead to the well-posedness of the direct problem.

If there exists a general incident field satisfying the Helmholtz equation in  $\R^d$ that does not produce a scattered field on the exterior of $D$, then, by \eqref{direct:1}--\eqref{direct:2}, one has that $w=u^s+u^i$ and $v=u^i$ are in $H^1(D)$ satisfying 
\begin{align}
    \Delta v + k^2v = 0 \quad &\text{ and } \quad \nabla \cdot A\nabla w+k^2nw=0 \quad \text{ in } D, \label{TE:1} \\ 
    w = v \quad &\text{ and } \quad \nu \cdot A\nabla w = \partial_\nu v + \eta v \quad \text{on $\partial D$.} \label{TE:2}
\end{align}
The values $k \in \mathbb{C}$ such that there exists a nontrivial solution $(w,v)\in H^1(D) \times H^1(D)$ to \eqref{TE:1}--\eqref{TE:2} are called transmission eigenvalues.

The rest of the paper is organized as follows. First, we prove that the transmission eigenvalues form an at most discrete set in the complex plane. This is done by considering an equivalent variational formulation of the problem and appealing to the analytic Fredholm theorem (see e.g. \cite{CCH-book}). We then prove the existence of infinitely many real transmission eigenvalues. It is shown that the eigenvalues are monotone with respect to the coefficients. This implies that the eigenvalues can be used to determine defects in the scatterer. The limiting cases as $\eta \to 0$ as well as $\eta \to \pm \infty$ are considered. In either case, the limiting value for the transmission eigenpairs can be determined. Lastly, we provide some numerical validation for the theoretical results. 

\section{Discreteness of Transmission Eigenvalues} \label{section-discr}
In this section, we study the discreteness of the transmission eigenvalues (TEVs) for an anisotropic scatterer with a conductive boundary condition. This problem has been studied for the case when $A = I$ in \cite{te-cbc,te-cbc2,te-cbc3}. Also, see \cite{two-eig-cbc,electro-cbc} the study of this problem associated with an electromagnetic scatterer. This result is useful in application since many qualitative reconstruction methods such as the factorization method \cite{fm-waveguide,regfm2,periodictevs,kirschbook} and generalized linear sampling methods \cite{GLSM,glsm-cracks,CCH-book} fail to reconstruct the scatterer $D$ from measured data if $k$ is a TEV. Here we study an equivalent variational formulation of \eqref{TE:1}--\eqref{TE:2} to treat the case of  $A \neq I$ and $\eta \neq 0$, which is not covered in \cite{te-cbc,CCH-book,te-geo-paper1,te-cbc2,te-cbc3}. We refer the readers to \cite{cavities,On-the-interior-TE,ayala2022analysis,te-geo-paper1,te-cbc3} for discreteness results of other TEV problems.   

We begin with the case when $A_{\text{min}}-1>0$ and $\eta_{\text{max}}<0$. Multiplying the second equation in \eqref{TE:1} by $\overline{\phi} \in H^1(D)$ and appealing to Green's 1st theorem, we have that 
\begin{equation*}
\int_D \nabla \overline{\phi} \cdot A\nabla w-k^2nw\overline{\phi} \,\text{d}x\ = \int_{\partial D} (\nu \cdot A\nabla w)\overline{\phi} \,\text{d}s.
\end{equation*}
By applying the boundary conditions in \eqref{TE:2} we have that
$$0=\int_D \nabla \overline{\phi} \cdot A\nabla w-k^2nw\overline{\phi} \,\text{d}x\ - \int_{\partial D} \overline{\phi}\big(T_kw+ \eta w \big) \,\text{d}s \quad \text{for all} \,\, \phi \in H^1(D),$$
where the Dirichlet-to-Neumann(DtN) map for the Helmholtz equation is defined as
\begin{equation*}
    T_k : H^{1/2}(\partial D) \longrightarrow  H^{-1/2}(\partial D) \quad \text{given by} \quad T_k f:=\partial_\nu v
\end{equation*}
such that 
\begin{equation*}
 \Delta v +k^2v = 0 \quad \text{in } D \quad \text{and} \quad v  = f \quad \text{on } \partial D.
\end{equation*}
Note that $T_k$ is well-defined for all complex values $k^2$ that are not Dirichlet eigenvalues of the negative Laplacian in $D$. Moreover, $T_k$ as an operator depends analytically on $k$ for all values for which it is well-defined \cite{BEM-FEMtransproblem}. 

We now define the bounded sesquilinear form $a_k(\cdot \, , \cdot ): H^1(D) \times H^1(D) \mapsto \C$ such that 
\begin{equation}
a_k(w,\phi):=\int_D \nabla\overline{\phi} \cdot A\nabla w -k^2nw\overline{\phi} \,\text{d}x\ - \int_{\partial D} \overline{\phi}\big(T_kw+ \eta w \big) \,\text{d}s. \label{varforma:1}
\end{equation}
It is clear that $k$ is a TEV if and only if there is a non-trivial $w$ satisfying $a_k(w,\phi)=0$ for all $\phi \in H^1(D)$, provided that $k^2$ is not a Dirichlet eigenvalue of the negative Laplacian in $D$. We denote the set of Dirichlet eigenvalues as $\{ \lambda_j (D)\}_{j=1}^\infty$. 

To prove the discreteness, we appeal to the analytic Fredholm theorem. We first show that the sesquilinear form associated with the transmission eigenvalue problem is represented by a Fredholm operator with index zero that depends on $k$ analytically in an open subset of the complex plane. To this end, let $a_k(\cdot,\cdot) = b(\cdot,\cdot) + c_k(\cdot,\cdot)$, where
\begin{equation}
b(w,\phi)= \int_D   \nabla\overline{\phi} \cdot A\nabla w  +A_{\text{min}}w\overline{\phi} \,\text{d}x\ - \int_{\partial D} \overline{\phi}T_ {\text{i}} w \,\text{d}s. \label{varformb:1}
\end{equation}
Here $T_{\text{i}} $ is the DtN mapping with $k = \text{i}$ and
\begin{equation}
    c_k(w,\phi)= -\int_D \big( k^2n+A_{\text{min}} \big) w \overline{\phi} \,\text{d}x\ - \int_{\partial D} \overline{\phi}(T_k-T_{\text{i}} )w+ \eta \overline{\phi}w  \,\text{d}s.\label{varformc:1}
\end{equation}
In the next few lemmas, we show that $a_k(\cdot,\cdot)$, $b(\cdot,\cdot)$, and $c_k(\cdot,\cdot)$ have analytical properties that allow us to use the analytic Fredholm theorem.

\begin{lemma}
Let $b(\cdot,\cdot)$  and $c_k(\cdot,\cdot)$ be given as in \eqref{varformb:1} and \eqref{varformc:1}, respectively. If we assume $A_{min}-1>0$ then $b(\cdot,\cdot)$ is coercive and  $c_k(\cdot,\cdot)$ is compact and analytic with respect to $k$ provided that $k^2 \in \C \setminus \{ \lambda_j (D)\}_{j=1}^\infty$.
\end{lemma}
\begin{proof}
We first estimate the boundary term in the variational form. By Green's 1st theorem and the fact that $v=w$ on $\partial D$, it holds that
 $$I = \int_{\partial D} \overline{w}T_{\text{i}} w \,\text{d}s\ = \int_{\partial D} \overline{v}\partial_\nu v \,\text{d}s\ = \int_{D} |\nabla v|^2+|v|^2 \,\text{d}x\ = ||v||^2_{H^1(D)} .$$
The Cauchy-Schwartz inequality implies that
$$I = \int_{\partial D} \overline{w}\partial_\nu v \,\text{d}s\ = \int_{D} \nabla\overline{w} \cdot \nabla v+\overline{w}v \,\text{d}x\leq ||v||_{H^1(D)}||w||_{H^1(D)} = \sqrt{I}||w||_{H^1(D)} .$$
Consequently, $I\leq ||w||^2_{H^1(D)}$.
Then $b(\cdot,\cdot)$ is coercive since
 $$b(w,w) = \int_D \nabla \overline{w} \cdot A\nabla w+A_{\text{min}}|w|^2 \,\text{d}x\ - \int_{\partial D} \overline{w}T_{\text{i}} w \,\text{d}s \geq \big(A_{\text{min}}-1 \big)||w||^2_{H^1(D)} .$$
    
 By the compact embeddings of $H^1(D)$ into $L^2(D)$ and $H^{1/2}(\partial D)$ into $L^2(\partial D)$ as well as the fact that $T_k-T_{\text{i}}$ is compact(see for e.g. \cite{BEM-FEMtransproblem}), there exists a compact operator $F_k :H^1(D) \to  H^1(D)$ such that $c_k(w,\phi)=(F_kw,\phi)_{H^1(D)}$ for all $w,\phi \in H^1(D)$ due to the Riesz representation theorem. Furthermore, $T_k$ is analytic provided that $k^2$ is not a  Dirichlet eigenvalue of the Laplacian in $D$. Consequently, $F_k$ is analytic with respect to $k^2 \in \C \setminus \{ \lambda_j (D)\}_{j=1}^\infty$, proving the claim.
\end{proof}

Since $b(\cdot,\cdot)$ is bounded and coercive, there exists an invertible operator $B:H^1(D) \to H^1(D)$ such that $b(w,\phi)=(Bw,\phi)_{H^1(D)}$ for all $w,\phi \in H^1(D)$. Hence the operator $B+F_k$ represents the sesquilinear form $a_k(\cdot \, , \cdot )$ which is Fredholm and analytic with respect to $k^2 \in  \C \setminus \{ \lambda_j (D)\}_{j=1}^\infty$. Next we show that  $a_k(\cdot \, , \cdot )$ is injective. 

\begin{lemma}
 Let $a_k(\cdot,\cdot)$ be given as in \eqref{varforma:1} with $A_{\text{min}}-1>0$ and $\eta_{\text{max}}<0$ then $a_0(\cdot,\cdot)$ is coercive (i.e. injective).
\end{lemma}
\begin{proof}
We first notice that for $k=0$
$$a_0(w,w)= \int_D \nabla \overline{w} \cdot A\nabla w \,\text{d}x\ - \int_{\partial D} \overline{w}T_0w+ \eta |w|^2 \,\text{d}s.$$
By the boundary condition $v=w$ on $\partial D$ we obtain that 
$$I = \int_{\partial D} \overline{w}T_0w \,\text{d}s\ = \int_{\partial D} \overline{v}\partial_\nu v \,\text{d}s\ = \int_{D} |\nabla v|^2 \,\text{d}x\ = ||\nabla v||^2_{L^2(D)} .$$
We conclude that $I\leq ||\nabla w||^2_{L^2(D)}$.
Therefore, $a_0(\cdot,\cdot)$ is coercive (injective) due to
$$a_0(w,w) \geq \big(A_{\text{min}}-1 \big)||\nabla w||^2_{L^2(D)} - \eta_{\text{max}}||w||^2_{L^2(\partial D)}.  $$
This proves coercivity by the equivalence of norms $||\cdot||^2_{H^1(D)}$ and  $||\nabla \cdot||^2_{L^2(D)} + ||\cdot||^2_{L^2(\partial D)}$ in $H^1(D)$(see for e.g. \cite{salsa}) as well as the fact that $\eta_{\text{max}}<0$.
\end{proof}

To continue, we now consider the case when $A_{\text{max}}-1<0$ and $\eta_{\text{min}}>0$. Similar to the previous case, we study the equivalent variational form for the TEV problem. Multiplying the first equation in \eqref{TE:1} by the conjugate of $\phi \in H^1(D)$ and appealing to Green's 1st theorem we have that 
 $$\int_D \nabla \overline{\phi} \cdot \nabla v-k^2v\overline{\phi} \,\text{d}x\ = \int_{\partial D} \overline{\phi}\partial_\nu v \,\text{d}s.$$
Define the DtN map
\begin{equation*}
 S_k : H^{1/2}(\partial D) \longrightarrow H^{-1/2}(\partial D) \quad \text{given by} \quad S_k f:=\nu \cdot A\nabla w
\end{equation*}
for the boundary value problem
\begin{equation*}
 \nabla A\nabla w +k^2 n w = 0 \quad \text{in } D \quad \text{and} \quad w  = f \quad \text{on } \partial D.
\end{equation*}
Provided that $k^2$'s  are not Dirichlet eigenvalues of $-n^{-1}\nabla \cdot A\nabla$ for $D$, the DtN mapping is a well-defined bounded linear operator. Denote by $\{ \lambda_j (D,A,n)\}_{j=1}^\infty$ the set of Dirichlet eigenvalues for the operator $-n^{-1}\nabla \cdot A\nabla$ for the region $D$. If $k^2 \in \C \setminus \{ \lambda_j (D,A,n)\}_{j=1}^\infty$ then $S_k$ is analytic with respect to $k$. 

We can now write $a_k(\cdot \, , \cdot ): H^1(D) \times H^1(D) \mapsto \C$ as
\begin{equation}
a_k(v,\phi):= \int_D \nabla \overline{\phi} \cdot \nabla v-k^2v\overline{\phi} \,\text{d}x\ - \int_{\partial D} \overline{\phi}\big(S_kv- \eta v \big) \,\text{d}s \label{varforma:2}
\end{equation}
and decomposed it into $a_k(\cdot,\cdot) = b(\cdot,\cdot) + c_k(\cdot,\cdot)$, where
\begin{equation}
    b(v,\phi)= \int_D \nabla \overline{\phi} \cdot \nabla v+v\overline{\phi} \,\text{d}x\ - \int_{\partial D} \overline{\phi}S_0v \,\text{d}s \label{varformb:2}
\end{equation}
and
\begin{equation}
    c_k(v,\phi)= -\int_D \big( k^2+1 \big) v\overline{\phi} \,\text{d}x\ - \int_{\partial D} \overline{\phi}(S_k-S_0)v- \eta v \overline{\phi}\,\text{d}s. \label{varformc:2}
\end{equation}
Note that both $b(\cdot,\cdot)$ and $c_k(\cdot,\cdot)$ are bounded sesquilinear forms on $H^1(D) \times H^1(D)$. In the next few lemmas, we show that $a_k(\cdot,\cdot)$ can be represented by an operator that is analytic with respect to $k^2 \in \C \setminus \{ \lambda_j (D,A,n)\}_{j=1}^\infty$ and Fredholm with index zero. 

\begin{lemma}
Let $b(\cdot,\cdot)$  and $c_k(\cdot,\cdot)$ be given as in \eqref{varformb:2} and \eqref{varformc:2}, respectively. If $A_{\text{max}}-1 < 0$ then $b(\cdot,\cdot)$ is coercive and, provided that $k^2 \in \C \setminus \{ \lambda_j (D,A,n)\}_{j=1}^\infty$, $c_k(\cdot,\cdot)$ is compact and analytic with respect to $k$. 
\end{lemma}
\begin{proof}
Since $A$ is a symmetric real-valued positive definite matrix, there exists $A^{1/2}$ such that $A=A^{1/2}\cdot A^{1/2}$. Using  Green's 1st theorem and the boundary condition $v=w$ on $\partial D$, it holds that
$$I = \int_{\partial D} \overline{v}S_0v \,\text{d}s\ = \int_{\partial D} \overline{w}(\nu \cdot A\nabla w) \,\text{d}s= \big(\nabla w, A\nabla w \big)_{L^2(D)} =||A^{1/2}\nabla w||^2_{L^2(D)} .$$
Therefore, 
$$I = \int_{\partial D} \overline{v}(\nu \cdot A\nabla w) \,\text{d}s\ = \big(\nabla v, A\nabla w \big)_{L^2(D)} = \big(A^{1/2}\nabla v, A^{1/2}\nabla w \big)_{L^2(D)}. $$
It holds that
$$I \leq ||A^{1/2}\nabla v||_{L^2(D)}||A^{1/2}\nabla w||_{L^2(D)} \leq \sqrt{A_{\text{max}}}||\nabla v||_{L^2(D)}||A^{1/2}\nabla w||_{L^2(D)}.$$ 
This implies that $I\leq A_{\text{max}}||\nabla v||^2_{L^2(D)}\leq A_{\text{max}}||v||^2_{H^1(D)}$.
Therefore, $b(\cdot,\cdot)$ is coercive since
$$b(v,v)= \int_D |\nabla v|^2+|v|^2 \,\text{d}x\ - \int_{\partial D} \overline{v}S_0v \,\text{d}s\ \geq \big(1-A_{\text{max}} \big)||v||^2_{H^1(D)} .$$

By the compact embeddings of $H^1(D)$ into $L^2(D)$ and $H^{1/2}(\partial D)$ into $L^2(\partial D)$ and the fact that $S_k-S_0$ is compact, there exists a compact operator $F_k: H^1(D) \to H^1(D)$ such that $c_k(v,\phi)=(F_kv,\phi)_{H^1(D)}$ for all $v,\phi \in H^1(D)$. We also have that $S_k$ is analytic, and thus $c_k(\cdot, \cdot)$ is analytic for all $k$ such that  $k^2 \in \C \setminus \{ \lambda_j (D,A,n)\}_{j=1}^\infty$. The proof is complete.
\end{proof}

Consequently, there exist an invertible operator $B$ and a compact operator $F_k$ which is Fredholm and analytic with respect to $k^2 \in  \C \setminus \{ \lambda_j (D,A,n)\}_{j=1}^\infty$ associating with the sesquilinear form $a_k(\cdot \, , \cdot )$. Now we prove the injectivity of $a_k(\cdot \, , \cdot )$ for $k^2 \in  \C \setminus \{ \lambda_j (D,A,n)\}_{j=1}^\infty$.

\begin{lemma}
Let $a_k(\cdot,\cdot)$ be given as in \eqref{varforma:2} with $A_{\text{max}}-1<0$ and $\eta_{\text{min}}>0$. Then $a_0(\cdot,\cdot)$ is coercive (i.e. injective).  
\end{lemma}
\begin{proof}
Note that
$$a_0(v,v)= \int_D |\nabla v| \,\text{d}x\ - \int_{\partial D} \overline{v}S_0v - \eta |v|^2 \,\text{d}s$$
and $I  \leq A_{\text{max}}||\nabla v||^2_{L^2(D)}$. Thus, $a_0(\cdot,\cdot)$ is coercive since
$$a_0(v,v) \geq \big(1-A_{\text{max}} \big)||\nabla v||^2_{L^2(D)} + \eta_{\text{min}}||v||^2_{L^2(\partial D)}$$
by the equivalence of $||\cdot||^2_{H^1(D)}$ and $||\nabla \cdot||^2_{L^2(D)} + ||\cdot||^2_{L^2(\partial D)}$ in $H^1(D)$.
\end{proof}

With this, we have proven that $a_k(\cdot,\cdot)$ is analytic and Fredholm for all complex-valued $k$ except for a discrete set on the real line. For either case, we have that $a_0(\cdot,\cdot)$ is an injective sesquilinear form. We now state the main theorem for this section. 

\begin{theorem}
If $A_{\text{min}}-1>0$ and $\eta_{\text{max}}<0$ or $A_{\text{max}}-1<0$ and $\eta_{\text{min}}>0$, then the set of transmission eigenvalues is at most  discrete.\label{TEVdiscrete}
\end{theorem}
\begin{proof}
This is a consequence of the fact that TEV problem is given by 
$$0=a_k(w,\phi)=\left((B+F_k)w,\phi \right)_{H^1(D)} \quad \text{ for all} \,\, \phi \in H^1(D)$$ 
where $B$ is invertible and $F_k$ is a compact operator and analytic for $k$ except for a discrete set. The result follows the analytic Fredholm theorem since $a_0(\cdot,\cdot)$ is an injective sesquilinear form. 
\end{proof}

\section{Existence of Transmission Eigenvalues} \label{section-exist}
In this section, we show that there exist infinitely many real TEVs. We shall use a technique that was first introduced in \cite{On-the-interior-TE} and used in \cite{ayala2022analysis,periodictevs} to prove the existence of TEVs associated with other scattering problems. Again, we remark that this problem has not yet been studied when $A \neq I$. We shall employ the assumptions that $A_{\text{min}}-1>0$, $n_{\text{max}}-1<0$, and $\eta_{\text{max}}<0$ or $A_{\text{max}}-1<0$, $n_{\text{min}}-1>0$, and $\eta_{\text{min}}>0$. Note that the discreteness needs no assumption on the refractive index $n$. Once the existence is established we show that the first TEV is monotone with respect to the coefficients. The result can be used to study of the inverse spectral problem of estimating the coefficients from the TEVs. 

To show the existence of real transmission eigenvalues, we consider an auxiliary problem for $u=w-v$ in $H^1_0(D)$. Note that, by \eqref{TE:1}--\eqref{TE:2}, $u$ satisfies
\begin{align}
    \nabla \cdot A\nabla u+k^2nu &= \nabla \cdot (I-A)\nabla v-k^2(n-1)v \quad \text{in $D$,} \label{TE:3}\\
    \nu \cdot A\nabla u &= \nu \cdot (I-A)\nabla v + \eta v \quad \text{on $\partial D$,} \label{TE:4}
\end{align}
where $\Delta v +k^2v = 0$ in $D$. Following \cite{On-the-interior-TE}, the above system \eqref{TE:3}--\eqref{TE:4} can be viewed as a boundary value problem for $v \in H^1(D)$ given $u \in H_0^1(D)$. We shall see that the system is well-posed for any $k \in \R_{\geq 0}$. Once we have this we will need to enforce that $v$ obtained by solving \eqref{TE:3}--\eqref{TE:4} is a solution to the Helmholtz equation in $D$.

The variational formulation of \eqref{TE:3}--\eqref{TE:4} is to find $v$ such that
\begin{align}
    \int_{D} \nabla \overline{\phi} \cdot (I-A)\nabla v + k^2(n-1)v\overline{\phi} \,\text{d}x\ &+ \int_{\partial D} \eta v\overline{\phi} \,\text{d}s   \nonumber \\
    &= \int_{D} \nabla \overline{\phi} \cdot A\nabla u - k^2nu\overline{\phi}  \,\text{d}x. \ \label{varform:utov}
\end{align}
Using the assumption on the coefficients that either $A_{\text{min}}-1>0$, $n_{\text{max}}-1<0$, and $\eta_{\text{max}}<0$ or $A_{\text{max}}-1<0$, $n_{\text{min}}-1>0$, and $\eta_{\text{min}}>0$, it can be seen that the left hand side is a coercive sesquilinear form for all $k \in \R_{\geq 0}$. The right hand side is a bounded conjugate linear functional whose norm is bounded by the norm of $u \in H^1_0(D)$. By applying to the Lax-Milgram lemma we have that \eqref{varform:utov} is well-posed, i.e., for any given $u \in H_0^1(D)$, there exists a unique solution $v_u \in H^1(D)$ satisfying \eqref{TE:3}--\eqref{TE:4}. The mapping $u \longmapsto v_u$ from $H_0^1(D)$ to $H^1(D)$ is a bounded linear operator. Now, we want the eigenfunction $v_u$ to satisfy the Helmholtz equation in $D$ i.e. 
$$0 = \int_{D} \nabla v_u \cdot \nabla \overline{\psi} - k^2v_u\overline{\psi} \,\text{d}x \quad \text{for all} \,\, \psi \in H_0^1(D).$$

We define a bounded linear operator $\mathbb{L}_k : H_0^1(D) \to H_0^1(D)$ such that
\begin{equation}
\big( \mathbb{L}_ku,\psi \big)_{H^1(D)} := \int_{D} \nabla v_u \cdot \nabla \overline{\psi} - k^2v_u\overline{\psi} \,\text{d}x \quad \text{for all} \,\, \psi \in H_0^1(D). \label{varformL:1}
\end{equation}
If there is a $k$ such that $\mathbb{L}_k u=0$ with $u\neq 0$ then $v_u$ is a solution to the Helmholtz equation in $D$. Letting $w=u+v_u$ we have that $(w,v_u) \in H^1(D) \times H^1(D)$ satisfies \eqref{TE:1}--\eqref{TE:2} with eigenvalue $k$. Conversely, if there is a TEV $k$ then $u=w-v$(nontrivial). In addition, $v$ satisfies \eqref{TE:3}--\eqref{TE:4} and is a solution for Helmholtz equation. From this, we deduce that $k \in \R_{\geq 0}$ is a TEV if and only if $\mathbb{L}_k$ has a nontrivial null space. 
The existence of real TEVs is equivalent to finding values when the null space of $\mathbb{L}_k$ is nontrivial. 

\begin{lemma}\label{Lkselfadjoint}
The operator $\mathbb{L}_k$ defined by \eqref{varformL:1} is self-adjoint.
\end{lemma}
\begin{proof}
    We shall show that $\big( \mathbb{L}_ku,u \big)_{H^1(D)}$ is real-valued for any $u \in H^1_0(D)$ which implies self-adjointness by Theorem 3:10-3 in \cite{func-analysis}. 
    Letting $\phi = u$ and $v=v_u$, \eqref{varform:utov} becomes
    \begin{equation}
        \int_{D} \nabla \overline{u} \cdot (I-A)\nabla v + k^2(n-1)v\overline{u} \,\text{d}x\ = \int_{D} \nabla \overline{u} \cdot A\nabla u - k^2n|u|^2   \,\text{d}x.  \label{Lself-adjoint:2}
    \end{equation}
Letting $\phi = v$, \eqref{varform:utov} becomes
\begin{align}
\int_{D} \nabla \overline{v} \cdot (I-A)\nabla v + k^2(n-1)|v|^2 \,\text{d}x\ &+ \int_{\partial D} \eta |v|^2 \,\text{d}s \nonumber\\
&= \int_{D} \nabla \overline{v} \cdot A\nabla u - k^2nu\overline{v}   \,\text{d}x. \label{Lself-adjoint:3}
\end{align}
Note that \eqref{varformL:1} can be written as 
\begin{align*}
\big( \mathbb{L}_ku,u \big)_{H^1(D)} = \int_{D} \nabla \overline{u} \cdot (I-A)\nabla v &+ k^2(n-1)v\overline{u} \,\text{d}x\ \\
&+ \int_{D} \nabla \overline{u} \cdot A\nabla v - k^2nv \overline{u}   \,\text{d}x.
\end{align*}
Using \eqref{Lself-adjoint:2} and \eqref{Lself-adjoint:3}, we have that 
\begin{align*}
\big( \mathbb{L}_ku,u \big)_{H^1(D)}= \int_{D} \nabla \overline{u} \cdot A\nabla u &- k^2n|u|^2   \,\text{d}x\ \\
&\hspace{-0.2in}+ \int_{D} \nabla \overline{v} \cdot (I-A)\nabla v + k^2(n-1)|v|^2 \,\text{d}x+ \int_{\partial D} \eta |v|^2 \,\text{d}s.
\end{align*}
From the fact that $A$ and $I-A$ are symmetric real-valued matrices, we can infer that the above expression is real-valued for any $u \in H^1_0(D)$. Therefore, $\mathbb{L}_k$ is self-adjoint.
\end{proof}

We now show that $\mathbb{L}_0$ is coercive if $A_{\text{max}}-1<0$ and $\eta_{\text{min}}>0$ and $-\mathbb{L}_0$ is coercive if $A_{\text{min}}-1>0$ and $\eta_{\text{max}}<0$. This is consistent with the results in the previous section where a different variational argument is used. 

\begin{lemma}\label{L0coersive}
The operator $\mathbb{L}_0$ is coercive if $A_{\text{max}}-1<0$ and $\eta_{\text{min}}>0$ and $-\mathbb{L}_0$ is coercive if $A_{\text{min}}-1>0$ and $\eta_{\text{max}}<0$.
\end{lemma}
\begin{proof}
We first consider $\mathbb{L}_0$. Recall in the proof of the previous lemma that
\begin{align*}
\big( \mathbb{L}_ku,u \big)_{H^1(D)} = \int_{D} \nabla \overline{u} \cdot A\nabla u &- k^2n|u|^2   \,\text{d}x\\
 &+ \int_{D} \nabla \overline{v} \cdot (I-A)\nabla v + k^2(n-1)|v|^2 \,\text{d}x\ + \int_{\partial D} \eta |v|^2 \,\text{d}s.
\end{align*}
It holds that
\begin{align*}
\big( \mathbb{L}_0u,u \big)_{H^1(D)} &= \int_{D} \nabla \overline{u} \cdot A\nabla u \,\text{d}x\ + \int_{D} \nabla \overline{v} \cdot (I-A)\nabla v \,\text{d}x\ + \int_{\partial D} \eta |v|^2 \,\text{d}s\ \\
&\geq A_{\text{min}}||\nabla u||^2_{L^2(D)} + (1-A_{\text{max}})||\nabla v||^2_{L^2(D)} + \eta_{\text{min}}||v||^2_{L^2(\partial D)} \\  
&\geq A_{\text{min}}||\nabla u||^2_{L^2(D)}.
\end{align*} 
Thus, by the Poincar\'{e} inequality, $\mathbb{L}_0$ is coercive if $A_{\text{max}}-1<0$ and $\eta_{\text{min}}>0$. \\

Next we show that $-\mathbb{L}_0$ is coercive. Letting $w=u+v$,  \eqref{varformL:1} becomes
    \begin{equation}
        \big( \mathbb{L}_ku,u \big)_{H^1(D)} = \int_{D} \nabla w \cdot \nabla \overline{u} - k^2w\overline{u} \,\text{d}x\ - \int_{D} |\nabla u|^2-k^2|u|^2 \,\text{d}x. \label{-L0coersive:1}
    \end{equation}
   Taking $\phi = w$ in \eqref{varform:utov}, we have that
$$\int_{D} \nabla \overline{w} \cdot (I-A)\nabla v + k^2(n-1)v\overline{w} \,\text{d}x\ + \int_{\partial D} \eta v\overline{w} \,\text{d}s\ = \int_{D} \nabla \overline{w} \cdot A\nabla u - k^2nu\overline{w}   \,\text{d}x.$$
It implies that
$$\int_{D} \nabla \overline{w} \cdot \nabla v - k^2v\overline{w} \,\text{d}x\ + \int_{\partial D} \eta |w|^2 \,\text{d}s\ = \int_{D} \nabla \overline{w} \cdot A\nabla w -k^2n |w|^2   \,\text{d}x, $$
where we use that $w=u+v$ and $w=v$ on $\partial D$. Then adding   
${\displaystyle \int_D \nabla \overline{w} \cdot \nabla u -k^2u\overline{w} \,\text{d}x}$
to both sides of the above equation and rearranging the terms, one obtains 
\begin{align}
\int_D \nabla \overline{w} \cdot \nabla u -k^2u\overline{w} \,\text{d}x\ = \int_D \nabla \overline{w} \cdot (I-A)\nabla w +k^2(n-1)|w|^2 \,\text{d}x +\int_{\partial D} \eta |w|^2 \,\text{d}s. \label{-L0coersive:2}
\end{align}
Using \eqref{-L0coersive:2}, we can rewrite \eqref{-L0coersive:1} as
\begin{align*}
\big( \mathbb{L}_ku,u \big)_{H^1(D)} = \int_D \nabla \overline{w} \cdot (I-A)\nabla w &+k^2(n-1)|w|^2 \,\text{d}x \nonumber \\
&+ \int_{\partial D} \eta |w|^2 \,\text{d}s\ - \int_{D} |\nabla u|^2-k^2|u|^2 \,\text{d}x. 
\end{align*}
We let $k=0$ and obtain
\begin{align*}
-\big(\mathbb{L}_0u,u \big)_{H^1(D)} &= \int_D \nabla \overline{w} \cdot (A-I)\nabla w \,\text{d}x\ - \int_{\partial D} \eta |w|^2 \,\text{d}s\ + \int_{D} |\nabla u|^2 \,\text{d}x \\
&\geq ||\nabla u||^2_{L^2(D)},
\end{align*}
  where we have used the fact that $A_{\text{min}}-1>0$ and $\eta_{\text{max}}<0$. Hence, by the Poincar\'{e} inequality, $-\mathbb{L}_0$ is coercive.
\end{proof}


\begin{lemma}\label{Lk-L0compact}
The operator $\mathbb{L}_k - \mathbb{L}_0$ is compact provided that either $A_{max}-1<0$ and $\eta_{min}>0$ or $A_{min}-1>0$ and $\eta_{max}<0$.
\end{lemma}
\begin{proof}
Assume that there exists a sequence $u^j \rightharpoonup 0$ in $H^1_0(D)$. By \eqref{varform:utov}, there exist sequences $v_k^j \rightharpoonup 0$ and $v_0^j \rightharpoonup 0$ in $H^1(D)$. These sequences correspond to the solutions of \eqref{varform:utov} for each given $k$. We now define $(\mathbb{L}_k - \mathbb{L}_0)u^j$ in terms of $v_k^j$ and $v_0^j$. To this end, using the variational form \eqref{varform:utov}, we have that, for all $\phi \in H^1(D)$,
\begin{align*}
\int_{D} \nabla \overline{\phi} \cdot (I-A)\nabla v_k^j + k^2(n-1)v_k^j\overline{\phi} \,\text{d}x\ &+ \int_{\partial D} \eta v_k^j\overline{\phi} \,\text{d}s\\
&= \int_{D} \nabla \overline{\phi} \cdot A\nabla u^j - k^2nu^j\overline{\phi} \,\text{d}x\
\end{align*}
    and
$$\int_{D} \nabla \overline{\phi} \cdot (I-A)\nabla v_0^j \,\text{d}x\ + \int_{\partial D} \eta v_0^j\overline{\phi} \,\text{d}s\ = \int_{D} \nabla \overline{\phi} \cdot A\nabla u^j \,\text{d}x.$$

We let $\phi = v_k^j-v_0^j$ and subtract the equations to obtain 
\begin{align}
\int_{D} \nabla \overline{(v_k^j-v_0^j)} \cdot (I-A)\nabla (v_k^j-v_0^j) \,\text{d}x\ &+ \int_{\partial D} \eta |v_k^j-v_0^j|^2 \,\text{d}s \nonumber\\
&\hspace{-0.5in}= \int_{D} -k^2nu^j\overline{(v_k^j-v_0^j)}-k^2(n-1)v_k^j\overline{(v_k^j-v_0^j)}   \,\text{d}x. \label{Lk-L0:compact}
\end{align}
The assumption that $A_{\text{max}}-1<0$ and $\eta_{\text{min}}>0$ or $A_{\text{min}}-1>0$ and $\eta_{\text{max}}<0$ implies that the left-hand side of \eqref{Lk-L0:compact} is equivalent to $||v_k^j-v_0^j||^2_{H^1(D)}$. This implies that 
$$||v_k^j-v_0^j||_{H^1(D)} \leq C\left( ||v_k^j||_{L^2(D)} + ||u^j||_{L^2(D)} \right), $$
where $C$ is independent of $j$. By the compact embedding of $H^1(D)$ into $L^2(D)$ it holds that
$$||v_k^j-v_0^j||_{H^1(D)} \longrightarrow 0 \quad \text{as } \quad j \to \infty .$$ 
By the definition of $\mathbb{L}_k$ in \eqref{varformL:1} we obtain that, for all $\psi \in H_0^1(D)$,
\begin{align*}
\big( (\mathbb{L}_k-\mathbb{L}_0)u^j,\psi \big)_{H^1(D)} = \int_{D} \nabla (v_k^j-v_0^j) \cdot \nabla \overline{\psi} - k^2v_k^j\overline{\psi} \,\text{d}x.
\end{align*}
Letting $\psi = (\mathbb{L}_k-\mathbb{L}_0)u^j$ and using the Cauchy-Schwartz inequality, we have that
$$||(\mathbb{L}_k-\mathbb{L}_0)u^j||_{H^1(D)} \leq \left( ||v_k^j-v_0^j||_{H^1(D)} +k^2 ||v_k^j||_{L^2(D)} \right) \longrightarrow 0 \quad \text{as } \quad j \to \infty .$$
Thus, $(\mathbb{L}_k-\mathbb{L}_0)u^j$ strongly converges to $0$ in $H^1(D)$, proving the compactness.
\end{proof}

From \eqref{varform:utov} and \eqref{varformL:1} we see that $\mathbb{L}_k$ is continuous with with respect to $k \in \R_{\geq 0}$. Next we prove that the operator $\pm \mathbb{L}_k$ is positive for certain $k$'s. 

\begin{lemma}\label{Lk:positive}
Let $\lambda_1(D)$ be the  first Dirichlet eigenvalue of $-\Delta$ in $D$ and $k^2$ be a real transmission eigenvalue. We have the following:
    \begin{enumerate}
        \item If $A_{\text{min}}-1>0$, $n_{\text{max}}-1<0$, and $\eta_{\text{max}}<0$, then $-\mathbb{L}_k$ is positive for $k^2 < \lambda_1(D)$.
        \item If $A_{\text{max}}-1<0$, $n_{\text{min}}-1>0$, and $\eta_{\text{min}}>0$, then $\mathbb{L}_k$ is positive for $k^2 < \frac{A_{\text{min}}}{n_{\text{max}}}\lambda_1(D)$.
    \end{enumerate}
\end{lemma}
\begin{proof}
    {1.} Let $A_{\text{min}}-1>0$, $n_{\text{max}}-1<0$, and $\eta_{\text{max}}<0$. From the proof of Lemma \ref{L0coersive},
    \begin{align*}
- \big(\mathbb{L}_ku,u \big)_{H^1(D)} = -\int_D \nabla \overline{w} \cdot (I-A)\nabla w &+k^2(n-1)|w|^2 \,\text{d}x \nonumber \\
 &- \int_{\partial D} \eta |w|^2 \,\text{d}s\ + \int_{D} |\nabla u|^2-k^2|u|^2 \,\text{d}x\ . 
    \end{align*}
    Using the Poincar\'{e} inequality $||u||^2_{L^2(D)} \leq \frac{1}{\lambda_1(D)}||\nabla u||^2_{L^2(D)}, u\in H_0^1(D),$ we obtain that
    \begin{equation*}
        -\big( \mathbb{L}_ku,u \big)_{H^1(D)} \geq \int_{D} |\nabla u|^2-k^2|u|^2 \,\text{d}x\ \geq \left( 1-\frac{k^2}{\lambda_1(D)} \right)||\nabla u||^2_{L^2(D)} .
    \end{equation*}
    Hence $-\mathbb{L}_k$ is positive for values of $k$ such that $k^2<\lambda_1(D) $.

    {2.} Let $A_{\text{max}}-1<0$, $n_{\text{min}}-1>0$, and $\eta_{\text{min}}>0$. From Lemma \ref{Lkselfadjoint}, 
    \begin{align*}
        \big( \mathbb{L}_ku,u \big)_{H^1(D)} = \int_{D} \nabla \overline{u} \cdot A\nabla u &- k^2n|u|^2 \,\text{d}x\\
        &+ \int_{D} \nabla \overline{v} \cdot (I-A)\nabla v + k^2(n-1)|v|^2 \,\text{d}x\ + \int_{\partial D} \eta |v|^2 \,\text{d}s.
    \end{align*}
    Again, using the Poincar\'{e} Inequality, we have that 
    \begin{align*}
        \big( \mathbb{L}_ku,u \big)_{H^1(D)} &\geq \int_{D} \nabla \overline{u} \cdot A\nabla u - k^2n|u|^2 \,\text{d}x\ \geq \int_{D}  A_{\text{min}}|\nabla u|^2 - k^2n_{\text{max}}|u|^2 \,\text{d}x\\
        &\geq \left( A_{\text{min}}-n_{\text{max}} \frac{k^2}{\lambda_1(D)}  \right)||\nabla u||^2_{L^2(D)} .
    \end{align*}
Thus $\mathbb{L}_k$ is positive for values of $k$ such that $k^2 < \frac{A_{\text{min}}}{n_{\text{max}}}\lambda_1(D)$.
\end{proof}

From the above lemma, all real TEVs must satisfy either Faber--Krahn type inequality
$$k^2 \geq \lambda_1(D)  \quad \text{provided that}\quad  A_{\text{min}}-1>0, \,\,\, n_{\text{max}}-1<0, \,\,\,\eta_{\text{max}}<0$$
or  
$$k^2 \geq \frac{A_{\text{min}}}{n_{\text{max}}}\lambda_1(D) \quad \text{ provided that} \quad  A_{\text{max}}-1<0, \,\, \, n_{\text{min}}-1>0, \,\,\,\eta_{\text{min}}>0.$$
It is consistent with the case when $\eta=0$.

\begin{lemma}
    There exists a $\tau>0$ such that $-\mathbb{L}_\tau$ with $A_{\text{min}}-1>0$, $n_{\text{max}}-1<0$, and $\eta_{\text{max}}<0$ or $\mathbb{L}_\tau$ with $A_{\text{max}}-1<0$, $n_{\text{min}}-1>0$, and $\eta_{\text{min}}>0$ is non-positive on some $N$--dimensional subspace of $H_0^1(D)$ for any $N\in \mathbb{N}$. \label{thm-nonpos}
\end{lemma}
\begin{proof}
    We construct a finite-dimensional subspace of $H_0^1(D)$ where $-\mathbb{L}_\tau$ is non-positive (the other case is similar) by considering small disjoint balls $B_{\epsilon} \subset D$.
    Assume that $A_{\text{min}}-1>0$, $n_{\text{max}}-1<0$, and $\eta_{\text{max}}<0$. Consider the ball $B_\epsilon$ of radius $\epsilon > 0$ such that $B_\epsilon \subset D$. Using Theorems 2.5 and 2.6 from \cite{On-the-interior-TE} there exist TEVs for
    \begin{align}
\Delta v_\epsilon + \tau^2v_\epsilon = 0 \quad &\text{ and }\quad \nabla \cdot A_{\text{min}}\nabla w_\epsilon+\tau^2n_{\text{max}}w_\epsilon=0 \quad \text{in $B_\epsilon$}, \label{ballTEV:1}\\
w_\epsilon = v_\epsilon \quad &\text{ and }\quad \nu \cdot A_{\text{min}}\nabla w_\epsilon = \partial_\nu v_\epsilon  \quad \text{on $\partial B_\epsilon$.}\label{ballTEV:2}
    \end{align}
Let $u_\epsilon = w_\epsilon - v_\epsilon$ associated with the auxiliary eigenvalue $\tau$. Similar calculation as in Lemma \ref{L0coersive} leads to
\begin{equation}
    0= \int_{B_\epsilon} |\nabla u_\epsilon|^2 - \tau^2|u_\epsilon|^2 \,\text{d}x
    + \int_{B_\epsilon} \nabla \overline{w_\epsilon} \cdot (A_{\text{min}}-1)\nabla w_\epsilon - \tau^2(n_{\text{max}}-1)|w_\epsilon|^2 \,\text{d}x. \label{existTEV:1}
\end{equation}
Since $u_\epsilon \in H_0^1(B_\epsilon)$, we can take an extension by zero to the whole domain $D$ and denote this function as $u \in H_0^1(D)$. Since $A_{\text{min}}-1>0$, $n_{\text{max}}-1<0$, and $\eta_{\text{max}}<0$, we can construct a nontrivial $v\in H^1(D)$ that solves \eqref{varform:utov} with coefficients $A$,$n$, and $\eta$ in $D$ where $w=u+v$.
Adding
$$\int_D \nabla u \cdot \nabla \overline{\phi} - \tau^2u\overline{\phi} \,\text{d}x$$
to both sides of \eqref{varform:utov} and rearranging the terms, we obtain that, for any $\phi \in H^1_0(D)$,
\begin{align}
    & \int_{D} \nabla \overline{\phi} \cdot (A-I)\nabla w - \tau^2(n-1)w\overline{\phi} \,\text{d}x\ - \int_{\partial D} \eta w \overline{\phi} \,\text{d}s \nonumber \\
    =& - \int_D \nabla u\cdot \nabla \overline{\phi} - \tau^2u\overline{\phi} \,\text{d}x \nonumber\\
    =& - \int_{B_\epsilon} \nabla u_\epsilon \cdot \nabla \overline{\phi} - \tau^2u_\epsilon\overline{\phi} \,\text{d}x\nonumber\\
    =& \int_{B_\epsilon} \nabla \overline{\phi} \cdot (A_{\text{min}}-1)\nabla w_\epsilon - \tau^2(n_{\text{max}}-1)w_\epsilon\overline{\phi} \,\text{d}x. \label{existTEV:1}
\end{align}
The last line is obtained by replacing $w_\epsilon$ by $u_\epsilon$ in \eqref{ballTEV:1}--\eqref{ballTEV:2}. Since $A_{\text{min}}-1>0$, $n_{\text{max}}-1<0$, the right-hand side of \eqref{existTEV:1} is an inner product on $H^1(B_\epsilon)$. Let $\phi = w$ in \eqref{existTEV:1} to obtain
\begin{align*}
    \int_{D} \nabla \overline{w} \cdot (A-I)\nabla w &- \tau^2(n-1)|w|^2 \,\text{d}x- \int_{\partial D} \eta |w|^2 \,\text{d}s\\
    &= \int_{B_\epsilon} \nabla \overline{w} \cdot (A_{\text{min}}-1)\nabla w_\epsilon - \tau^2(n_{\text{max}}-1)w_\epsilon\overline{w} \,\text{d}x.
\end{align*}
Using the Cauchy-Schwartz inequality and $\eta_{\text{max}}<0$, we obtain that
\begin{align}
    \int_{D} \nabla \overline{w} \cdot (A-I)\nabla w &- \tau^2(n-1)|w|^2 \,\text{d}x- \int_{\partial D} \eta |w|^2 \,\text{d}s\nonumber \\
    &\leq \int_{B_\epsilon} \nabla \overline{w_\epsilon} \cdot (A_{\text{min}}-1)\nabla w_\epsilon - \tau^2(n_{\text{max}}-1)|w_\epsilon|^2 \,\text{d}x. \label{existTEV:2}
\end{align}
From the proof of Lemma \ref{L0coersive},
\begin{align*}
 -\big(\mathbb{L}_\tau u,u \big)_{H^1(D)} = \int_D \nabla \overline{w} \cdot (A-I)\nabla w &- \tau^2(n-1)|w|^2 \,\text{d}x - \int_{\partial D} \eta |w|^2 \,\text{d}s \\
        &+ \int_{D} |\nabla u|^2-\tau^2|u|^2 \,\text{d}x .
 \end{align*}
By the construction of $u$ and using  \eqref{existTEV:1}-\eqref{existTEV:2} we have that
    \begin{align*}
       - \big(\mathbb{L}_\tau u,u \big)_{H^1(D)} &= \int_D \nabla \overline{w} \cdot (A-I)\nabla w - \tau^2(n-1)|w|^2 \,\text{d}x  - \int_{\partial D} \eta |w|^2 \,\text{d}s \\
        &\hspace{+1in} + \int_{B_{\epsilon}} |\nabla u_\epsilon|^2-\tau^2|u_\epsilon|^2 \,\text{d}x\ \\ 
        &\leq \int_{B_\epsilon} \nabla \overline{w_\epsilon} \cdot (A_{\text{min}}-1)\nabla w_\epsilon - \tau^2(n_{\text{max}}-1)|w_\epsilon|^2 \,\text{d}x\ \\
        &\hspace{+1in} + \int_{B_{\epsilon}} |\nabla u_\epsilon|^2-\tau^2|u_\epsilon|^2 \,\text{d}x = 0.
    \end{align*}
Thus the operator $-\mathbb{L}_\tau$ is non-positive on this one--dimensional subspace. 

To construct an $N$--dimensional subspace for any $N \in \N$ we let $B_j$ be the ball centered at $x_j \in D$ with radius $\epsilon >0$. Let $\epsilon$ be sufficiently small such that $B_i \cap B_j = \emptyset$ for all $i \neq j$. Then, consider \eqref{ballTEV:1}--\eqref{ballTEV:2} in each $B_j$ and let $u_j \in H^1_0(D)$ be the difference of the eigenfunctions extended to $D$ by zero with eigenvalue $\tau$. Note that $\text{span}\{ {u}_j : \, j=1, \cdots , N\}$ is an $N$--dimensional subspace of $H^1_{0}(D)$ due to the disjoint supports. Hence $-\mathbb{L}_{\tau}$ is non-positive on this $N$--dimensional subspace of $H^1_{0}(D)$. The case for $\mathbb{L}_\tau$  can be handled similarly.
\end{proof}

The following theorem states that there exists infinitely many transmission eigenvalues.
\begin{theorem}\label{TEVexist}
    Assume that either $A_{\text{min}}-1>0$, $n_{\text{max}}-1<0$, and $\eta_{\text{max}}<0$, or  $A_{\text{max}}-1<0$, $n_{\text{min}}-1>0$, and $\eta_{\text{min}}>0$. Then there exists infinitely many real transmission eigenvalues $k > 0$. \label{thm-exist}
\end{theorem}
\begin{proof}
The theorem is a direct consequence of the lemmas proved in this section and Theorem 2.6 of \cite{On-the-interior-TE} (see also \cite{CCH-book}). 
\end{proof}

\section{Dependence on the parameters}
In this section, we study how the TEVs depend on the material parameters. It has been shown in previous works \cite{cavities,cakoni2014homogenization,te-cbc3} that the TEVs can be used to estimate or determine the material parameters associated with the scattering problem. This falls under the category of inverse spectral problems where one wishes to infer information about the physical model from the knowledge of an associated eigenvalue problem. We begin by proving the monotonicity of the first (i.e. smallest positive) TEV with respect to the coefficients. Then we consider limiting problems as $\eta \to 0$ and $\eta \to \pm \infty.$ Here, the limits are understood in the sense that there are sequences of $\eta_j, j \in \N$ such that 
$$ \| \eta_j \|_{L^\infty (\partial D)} \longrightarrow 0 \quad \text{or } \quad \inf\limits_{\partial D} | \eta_j | \longrightarrow \infty \quad \text{as } \quad j \to \infty .$$ 
See e.g. \cite{ayala2022analysis,te-cbc2,te-cbc3,LKtelimit} for other studies of limiting values for some TEV problems.   

\subsection{Monotonicity Results}\label{TEVmono}
We consider the monotonicity of the first TEV, denoted by $k_1$, with respect to the parameters $A,n$ and $\eta$. Note that  $k_1>0$ is the smallest value in $\R_{\geq 0}$ where the function $f(k) = \min_{u \neq 0}(\mathbb{L}_k u , u)_{H^1(D)}$ is non-positive (or non-negative) by Lemma \ref{Lk:positive} and the continuity of the mapping $k \mapsto \mathbb{L}_k$ from $\R_{\geq 0}$ to the set of bounded linear operators from $H^1_0(D)$ into itself. In our analysis, we say that $A_1 \leq A_2$, for matrix valued functions $A_1$ and $A_2$, if $A_2-A_1$ is semi-positive definite.

\begin{theorem}
    Assume that $A_{j}-I>0$, $n_{j}-1<0$, and $\eta_{j}<0$ for $j=1,2$. 
    \begin{enumerate}
        \item[(1)] If $n_1 \leq n_2$ and $\eta_1 \leq \eta_2$ then $k_1(n_1 , \eta_1) \leq k_1(n_2 , \eta_2)$.
        \item[(2)] If $A_1 \leq A_2$  then $k_1(A_2) \leq k_1(A_1)$.
    \end{enumerate}
\label{thm-mono1}
\end{theorem}
\begin{proof}
    We shall prove (1). Claim (2) can be proved in a similar way. If $n_1 \leq n_2$, then $-(n_1-1) \geq -(n_2-1) > 0$. Let  $\tau=k_1(A,n_2,\eta_2)$ and assume that $A>I$ with $\eta_{2}<0$, and $n_2-1<0$. Letting $v_\tau$ and $w_\tau$ denote the transmission eigenfunctions corresponding to the transmission eigenvalue $\tau$. From the proof of Lemma \ref{L0coersive} we have that 
    \begin{align*}
    0&=\int_{D} |\nabla u_\tau|^2 - \tau^2|u_\tau|^2 \,\text{d}x + \int_{D} \nabla \overline{w_\tau} \cdot (A-I)\nabla w_\tau - \tau^2(n_2-1)|w_\tau|^2 \, \text{d}x\ - \int_{\partial D} \eta_2 |w_\tau|^2 \,\text{d}s
    \end{align*}
    where $u_\tau = w_\tau-v_\tau$. 
    
By the well-posedness of \eqref{varform:utov}, we have the existence of $v \in H^1(D)$ with $u=u_\tau$, $n=n_1$, $\eta=\eta_1$ and $k=\tau$ where $w=v+u_\tau$. From \eqref{varform:utov} along with similar calculations as in Lemma \ref{thm-nonpos}, we have 
\begin{align*}
\int_{D} \nabla \overline{\phi} \cdot (A-I)\nabla w &- \tau^2(n_1-1)w\overline{\phi} \,\text{d}x- \int_{\partial D} \eta_1 w \overline{\phi} \,\text{d}s\\
        & = - \int_D \nabla u_\tau \cdot \nabla \overline{\phi} - \tau^2 u_\tau \overline{\phi} \,\text{d}x\\
        &= \int_{D} \nabla \overline{\phi} \cdot (A-I)\nabla w_\tau - \tau^2(n_2-1)w_\tau\overline{\phi} \,\text{d}x- \int_{\partial D} \eta_2 w_\tau \overline{\phi} \,\text{d}s.
    \end{align*}
    The right-hand side of the above equality is an inner-product on $H^1(D)$ due to the assumptions on the coefficients. Letting $\phi = w$, it holds that
    \begin{align*}
        \int_{D} \nabla \overline{w} \cdot (A-I)\nabla w &- \tau^2(n_1-1)|w|^2\,\text{d}x - \int_{\partial D} \eta_1 |w|^2 \,\text{d}s\\
        &= \int_{D} \nabla \overline{w} \cdot (A-I)\nabla w_\tau - \tau^2(n_2-1)w_\tau \overline{w} \,\text{d}x- \int_{\partial D} \eta_2 w_\tau \overline{w} \,\text{d}s.
    \end{align*}
We employ the Cauchy-Schwartz inequality to obtain
    \begin{align*}
        \int_{D} \nabla \overline{w} \cdot (A-I)\nabla w &- \tau^2(n_1-1)|w|^2\,\text{d}x- \int_{\partial D} \eta_1 |w|^2 \,\text{d}s\\
        &\leq \int_{D} \nabla \overline{w_\tau} \cdot (A-I)\nabla w_\tau - \tau^2(n_2-1)|w_\tau|^2 \,\text{d}x- \int_{\partial D} \eta_2 |w_\tau|^2 \,\text{d}s.
    \end{align*}
We have used that $-(n_1-1) \geq -(n_2-1) > 0$ as well as $-\eta_1 \geq - \eta_2>0$.

Denote by $\mathbb{L}_\tau$ the operator with coefficients $A,n_1$ and $\eta_1$. It holds that
    \begin{align*}
        -\big(\mathbb{L}_{\tau}u_\tau ,u_\tau \big)_{H^1(D)} &= \int_D \nabla \overline{w} \cdot (A-I)\nabla w - \tau^2(n_1-1)|w|^2 \,dx - \int_{\partial D} \eta_1 |w|^2 \,ds\\
        &\hspace{1in}+ \int_{D} |\nabla u_\tau|^2-\tau^2|u_\tau|^2 \,dx\ \\
        &\leq \int_D \nabla \overline{w_\tau} \cdot (A-I)\nabla w_\tau - \tau^2(n_2-1)|w_\tau|^2 \,dx - \int_{\partial D} \eta_2 |w_\tau|^2 \,ds\ \\
        &\hspace{1in}+ \int_{D} |\nabla u_\tau|^2-\tau^2|u_\tau|^2 \,dx\ = 0 .
    \end{align*}
This implies that $-\mathbb{L}_{\tau}$ associated with $A,n_1$ and, $\eta_1$ is nonpositive on this one-dimensional subspace spanned by $u_\tau$. Then the continuous dependence of $k \mapsto \mathbb{L}_{k}$ implies the existence of a TEV corresponding to $A,n_1$ and $\eta_1$ in $(0, \tau]$ with $\tau=k_1(A,n_2,\eta_2)$. The proof is complete. 
\end{proof}

Similar monotonicity results hold for $A_{\text{max}}-1<0$, $n_{\text{min}}-1>0$, and $\eta_{\text{min}}>0$. 
\begin{theorem}
    Assume that $A_{j}-I<0$, $n_{j}-1>0$, and $\eta_{j}>0$ for $j=1,2$. 
    \begin{enumerate}
        \item If $n_1 \leq n_2$ and $\eta_1 \leq \eta_2$ then we have that  $k_1(n_2 , \eta_2) \leq k_1(n_1 , \eta_1)$.
        \item If $A_1 \leq A_2$ then we have that  $k_1(A_1) \leq k_1(A_2)$.
    \end{enumerate}
    where $k_1$ corresponds to the first transmission eigenvalue. \label{thm-mono2}
\end{theorem}
\begin{proof}
    The proof is similar to that of Theorem \ref{thm-mono1} and thus omitted.
\end{proof}

\subsection{Convergence of the TEVs as $\eta \longrightarrow 0$} \label{sec-etato0}
Now we give a more in depth study of the TEVs on the dependence on $\eta$. In particular, we are interested in the limiting behavior as $\eta \to 0$. Let $k_\eta$ correspond to  a TEV for $\eta \neq 0$ and assume that it forms a bounded set. Then $k_\eta$ has a limit point, denoted by $k_0$, as $\eta \to 0$. By Lemma \ref{Lk:positive} and Theorems \ref{thm-mono1} and \ref{thm-mono2}, $k_\eta$ converges as $\eta \to 0$ monotonically. 

Define the Hilbert space 
$$X(D):= \left\{ (\phi_1 , \phi_2) \in H^1(D) \times H^1(D) \,\, : \,\, \phi_1 = \phi_2 \,\, \text{on} \,\, \partial D \right\}$$
with the standard norm.
Assume that the non-trivial eigenfunctions $(w_\eta , v_\eta )\in X(D)$ satisfying \eqref{TE:1}--\eqref{TE:2} are normalized in the $L^2(D)\times L^2(D)$ norm such that 
 $$||w_\eta||^2_{L^2(D)} + ||v_\eta||^2_{L^2(D)} = 1$$
for any $\eta \neq 0$. We show that the eigenfunctions are bounded in $X(D)$ as $\eta \to 0$.

\begin{lemma} \label{bounded:etato0}
Assume that either $A_{\text{min}}-1>0$ or $A_{\text{max}}-1<0$ and that $k_\eta$ is bounded then $(w_\eta , v_\eta ) \in X(D)$ is bounded as $\eta \to 0$.
\end{lemma}
\begin{proof}
    We assume that $(w_\eta, v_\eta)$ satisfies \eqref{TE:1}--\eqref{TE:2} and let $(\phi_1, \phi_2 ) \in X(D)$. Then, by Green's first theorem, we have that  
\begin{align}
\int_{D} A\nabla w_\eta \cdot \nabla \overline{\phi_1} - k_\eta^2nw_\eta \overline{\phi_1} \,\text{d}x\ &= \int_{\partial D} \overline{\phi_1}(\nu \cdot A\nabla w_\eta) \, \text{d}s, \quad \text{and}  \label{varform:bounded1}\\
\int_{D} \nabla v_\eta \cdot \nabla \overline{\phi_2} - k_\eta^2v_\eta \overline{\phi_2} \,\text{d}x &= \int_{\partial D} \overline{\phi_2}\partial_\nu v_\eta \,\text{d}s \label{varform:bounded2}
\end{align}
Subtracting the equations \eqref{varform:bounded1}--\eqref{varform:bounded2}  and using the boundary conditions \eqref{TE:2} we get that 
$$\int_{D} A\nabla w_\eta \cdot \nabla \overline{\phi_1} - \nabla v_\eta \cdot \nabla \overline{\phi_2} - k_\eta^2nw_\eta \overline{\phi_1} + k_\eta^2v_\eta \overline{\phi_2} \,\text{d}x = \int_{\partial D} \eta v_\eta  \overline{\phi_2} \,\text{d}s.$$

(1) Consider the case when $A_{\text{min}}-1>0$. Write the above variational formulation as 
$$a\left( (w_\eta, v_\eta) ; (\phi_1 , \phi_2) \right) = b\left( (w_\eta, v_\eta) ;(\phi_1 , \phi_2) \right) \quad \text{for any} \quad (\phi_1 , \phi_2)  \in X(D)$$
where 
\begin{equation*}
a\left( (w_\eta, v_\eta) ; (\phi_1 , \phi_2) \right) = \int_{D} A\nabla w_\eta \cdot \nabla \overline{\phi_1} +A_{\text{min}}w_\eta \overline{\phi_1} - \nabla v_\eta \cdot \nabla \overline{\phi_2} - v_\eta \overline{\phi_2} \,\text{d}x\ 
\end{equation*}
and 
\begin{equation*}
b\left( (w_\eta, v_\eta) ;(\phi_1 , \phi_2)  \right) = \int_{D} ( nk_\eta^2+A_{\text{min}} )w_\eta \overline{\phi_1} - ( k_\eta^2+1 )v_\eta \overline{\phi_2} \,\text{d}x\ + \int_{\partial D} \eta v_\eta \overline{\phi_2} \,\text{d}s.
\end{equation*}
We first show the boundedness. To this end, if $\phi_1 =w_\eta$ and $\phi_2 = -v_\eta +2w_\eta $, then $(\phi_1 , \phi_2)  \in X(D)$. Hence
\begin{align*}
\big| a\left( (w_\eta, v_\eta) ; (w_\eta , -v_\eta +2w_\eta) \right) \big| &\geq \int_{D} A\nabla w_\eta \cdot \nabla \overline{w_\eta} +A_{\text{min}}|w_\eta|^2 + |\nabla v_\eta|^2 + |v_\eta|^2 \,\text{d}x\ \\
        &\hspace{1in} - 2\Bigg| \int_{D} \nabla v_\eta \cdot \nabla \overline{w_\eta} +v_\eta \overline{w_\eta} \,\text{d}x\ \Bigg| \\
        &\geq \Big( 1-\frac{1}{\delta} \Big)||v_\eta||^2_{H^1(D)} + \Big( A_{\text{min}} -\delta \Big)||w_\eta||^2_{H^1(D)}, 
\end{align*}
where the last line uses Young's inequality.  So for any $\delta \in (1, A_{\text{min}})$ we have that 
$$\big|a\left( (w_\eta, v_\eta) ; (w_\eta , -v_\eta +2w_\eta) \right)\big| \geq \alpha \left( ||v_\eta||^2_{H^1(D)} +||w_\eta||^2_{H^1(D)} \right),$$
where $\alpha>0$ is independent of $\eta$. For $b(\cdot, \cdot)$, it holds that
\begin{align*}
\big| b\left( (w_\eta, v_\eta) ; ( w_\eta , -v_\eta +2w_\eta) \right) \big| & \leq C_1 \Big[ ||v_\eta||^2_{L^2(D)} + ||w_\eta||^2_{L^2(D)} \Big] + \big|(\eta v_\eta, v_\eta)_{L^2(\partial D)} \big|\\
&\leq C_1 + C_2 \|\eta \|_{L^\infty (\partial D)} \left( ||v_\eta||^2_{H^1(D)} +||w_\eta||^2_{H^1(D)} \right),
\end{align*}
where the constants $C_j>0$ for $j=1,2$ are independent of $\eta$. Note that we have used the normalization of the eigenfunctions and the Trace Theorem. He it holds that  
$$\alpha \left( ||v_\eta||^2_{H^1(D)} +||w_\eta||^2_{H^1(D)} \right) \leq C_1 + C_2 \|\eta \|_{L^\infty (\partial D)} \left( ||v_\eta||^2_{H^1(D)} +||w_\eta||^2_{H^1(D)} \right),$$
which implies that 
$$ ||w_\eta||^2_{H^1(D)} + || v_\eta||^2_{H^1(D)} \leq \frac{C_1}{\alpha - C_2\|\eta \|_{L^\infty (\partial D)} } \quad \text{for}  \quad \eta \ll1.$$
We conclude that $||w_\eta||^2_{H^1(D)} + ||v_\eta||^2_{H^1(D)}$ is bounded as $\eta \to 0$.
    
(2) For  $A_{\text{max}}-1<0$, we rewrite \eqref{varform:bounded1}--\eqref{varform:bounded2} as
 \begin{align*}
\int_{D} A\nabla w_\eta \cdot \nabla \overline{\phi_1} +A_{\text{max}}w_\eta \overline{\phi_1} &- \nabla v_\eta \cdot \nabla \overline{\phi_2} - v_\eta \overline{\phi_2} \,\text{d}x\\
        &\hspace{-0.5in}= \int_{D} ( nk_\eta^2+A_{\text{max}} )w_\eta \overline{\phi_1} - ( k_\eta^2+1 )v_\eta \overline{\phi_2} \,\text{d}x\ + \int_{\partial D} \eta v_\eta \overline{\phi_2} \,\text{d}s.
    \end{align*}
Then the rest of the proof is similar to the previous case by taking $\phi_1 = w_\eta-2v_\eta$ and $\phi_2 = -v_\eta $ such that  $(\phi_1 , \phi_2)  \in X(D)$.
\end{proof}

We now have that $\left\{ k_\eta ; (w_\eta,v_\eta) \right\}  \in \mathbb{R}_{\geq 0} \times X(D)$ is bounded and thus there exist $\left\{ k_0 ; (w_0,v_0) \right\}  \in \mathbb{R}_{\geq 0} \times X(D)$ such that, up to a subsequence,
$$k_\eta \longrightarrow k_0 \quad \text{ as } \quad \eta \to 0$$
and
$$w_\eta \rightharpoonup w_0 \hspace{.2cm} \text{ and } \hspace{.2cm} v_\eta \rightharpoonup v_0 \hspace{.2cm} \text{ in $H^1(D)$ } \quad \text{as $\eta \to 0$. }$$
This implies that $(v_0,w_0)$ is nontrivial by appealing to the weak convergence and the compact embedding of $H^1(D)$ into $L^2(D)$. Moreover,  
$||w_0||^2_{L^2(D)} + ||v_0||^2_{L^2(D)} = 1.$

We now show that $\left\{ k_0 ; (w_0 ,v_0) \right\}$ satisfies \eqref{TE:1}--\eqref{TE:2}. 
Note that, by the weak convergence of the eigenfunctions, for any $\phi \in H^1(D)$,
$$\int_{\partial D} \overline{\phi}\Big[ \nu \cdot A\nabla w_\eta - \partial_\nu v_\eta \Big] \,\text{d}s \longrightarrow \int_{\partial D} \overline{\phi}\Big[ \nu \cdot A\nabla w_0 - \partial_\nu v_0 \Big] \,\text{d}s  \quad \text{ as } \quad \eta \to 0.$$
Hence
\begin{align*}
    \Bigg| \int_{\partial D} \overline{\phi}\Big[ \nu \cdot A\nabla w_\eta - \partial_\nu v_\eta \Big] \,\text{d}s\ \Bigg|  & =  \Bigg| \int_{\partial D} \eta v_\eta \overline{\phi} \,\text{d}s \Bigg| \\
    &\leq  \| \eta \|_{L^{\infty}(\partial D)} ||v_\eta||_{H^{1/2}(\partial D)}||\phi||_{H^{1/2}(\partial D)}\\
    &\leq C\| \eta \|_{L^{\infty}(\partial D)} ||\phi||_{H^{1/2}(\partial D)},
    \end{align*}
where we have used the boundedness of the eigenfunctions. This implies that 
$$\| \nu \cdot A\nabla w_\eta - \partial_\nu v_\eta \|_{H^{-1/2}(\partial D)} \longrightarrow 0  \quad \text{ as } \quad \eta \to 0$$
and we can conclude that $\nu \cdot A\nabla w_0 = \partial_\nu v_0 $ on $\partial D$. Due to the fact that the eigenfunctions $(w_\eta,v_\eta) \in X(D)$ satisfy \eqref{TE:1}--\eqref{TE:2} and the weak convergence as $\eta \to 0$, it holds that 
$$ \nabla \cdot A\nabla w_0 + k_0^2nw_0 = 0 \quad \text{and} \quad \Delta v_0 + k_0^2v_0 = 0 \quad \text{ in $D$} .$$
Consequently, $(w_0 ,v_0) \in X(D)$ is a non-trivial solution to \eqref{TE:1}--\eqref{TE:2} with eigenvalue $k_0$. We thus have shown the main result of this section. 

\begin{theorem}
Assume that either $A_{\text{min}}-1>0$ or $A_{\text{max}}-1<0$ and that $k_\eta$ is bounded then for $k_\eta \to k_0$ as $\eta \to 0$ it holds that $(w_\eta,v_\eta)\rightharpoonup  (w_0,v_0)$ in $X(D)$ as $\eta \to 0$ such that 
\begin{align*}
    \Delta v_0 + k_0^2v_0 = 0 \quad &\text{ and } \quad \nabla \cdot A\nabla w_0+k_0^2nw_0=0 \quad \text{ in } D, \\ 
    w _0= v_0 \quad &\text{ and } \quad \nu \cdot A\nabla w_0 = \partial_\nu v_0  \quad \text{on $\partial D$} 
\end{align*}
with $(w_0,v_0)$ being nontrivial. 
\label{thm-etato0}
\end{theorem}

\subsection{Convergence of the TEVs as $\eta \longrightarrow \pm \infty$} \label{sec-etatoinfty}
In this section, we consider the case when $\eta \to \pm \infty$. One will see that there is an interesting bifurcation for this limit. We shall make additional assumptions that the boundary $\partial D$ is of class $\mathscr{C}^2$ or is polygonal with no reentrant corners and $A \in \mathscr{C}^1(\overline{D},\R^{d \times d})$ so that we can use the standard elliptic regularity \cite{salsa}. Let $k_\eta$ denote the TEV with respect to $\eta$. We assume that $k_\eta$ is bounded as $\eta \to \pm \infty$ and therefore has a limit point. Again, by Lemma \ref{Lk:positive} as well as Theorems \ref{thm-mono1} and \ref{thm-mono2}, in some cases, $k_\eta$ has a limit point as $\eta \to \pm \infty$ monotonically.

Assume that the eigenfunctions $(w_\eta , v_\eta )$ satisfying \eqref{TE:1}--\eqref{TE:2} are normalized in $X(D)$ such that
$$||v_\eta||^2_{H^1(D)} + ||w_\eta||^2_{H^1(D)} = 1$$
for any $|\eta|$ finite.  
We have that $\left\{ k_\eta ; (w_\eta,v_\eta) \right\}  \in \mathbb{R}_{\geq 0} \times X(D)$ is bounded and thus there exists $\left\{ k_\infty ; (w_\infty , v_\infty) \right\}  \in \mathbb{R}_{\geq 0} \times X(D)$ such that, up to a subsequence,
$$k_\eta \longrightarrow k_\infty \quad \text{ as } \quad \eta \to \pm \infty$$
and
$$w_\eta \rightharpoonup w_\infty \hspace{.2cm} \text{ and } \hspace{.2cm} v_\eta \rightharpoonup v_\infty \hspace{.2cm} \text{ in $H^1(D)$ } \quad \text{as $\eta \to \pm \infty$. }$$
As in section \ref{sec-etato0}, the weak limit $(w_\infty , v_\infty) \in X(D)$ satisfies
\begin{equation*}
    \Delta v_\infty + k_\infty^2v_\infty = 0 \quad \text{ and }\quad \nabla \cdot A\nabla w_\infty+k_\infty^2nw_\infty =0 \quad \text{in $D$} 
\end{equation*}
by the weak convergence and Green's first theorem. 

To determine the limiting boundary value problem, we first show that the weak limits have zero trace on $\partial D$. By the conductivity condition in \eqref{TE:2}, we have that 
\begin{equation*}
\int_{\partial D} \overline{\phi}\Big[ \nu \cdot A\nabla w_\eta - \partial_\nu v_\eta \Big] \, \text{d}s\ = \int_{\partial D} \eta v_\eta \overline{\phi} \,\text{d}s\quad \text{ for all } \quad \phi \in H^1(D).
\end{equation*}
Letting $\phi = {v_\eta}$, the following estimate holds 
\begin{equation*}
\inf\limits_{\partial D} | \eta |  \int_{\partial D} |v_\eta|^2 \,\text{d}s \leq \left| \int_{\partial D} \overline{v_\eta}\left[ \nu \cdot A\nabla w_\eta - \partial_\nu v_\eta \right] \,\text{d}s\ \right|,
\end{equation*}
which implies that 
\begin{equation*}
\| v_\eta \|^2_{L^2(\partial D)} \leq \frac{1}{\inf\limits_{\partial D} | \eta |} ||\nu \cdot A\nabla w_\eta - \partial_\nu v_\eta||_{H^{-1/2}(\partial D)} ||v_\eta ||_{H^{1/2}(\partial D)}.
\end{equation*}
By the trace theorems and the normalization of the transmission eigenfunctions, we have that there exist $C_j>0$, $j=1,2$, independent of $\eta$ such that
\begin{equation*}
     ||\nu \cdot A\nabla w_\eta - \partial_\nu v_\eta||_{H^{-1/2}(\partial D)} \leq C_1 \quad \text{ and }\quad  ||v_\eta ||_{H^{1/2}(\partial D)} \leq C_2.
\end{equation*}
Therefore, we obtain that 
\begin{equation*}
    ||v_\eta ||^2_{L^2(\partial D)} \leq \frac{C}{ \inf\limits_{\partial D} | \eta |} \longrightarrow 0  \quad \text{ as }\quad \eta \longrightarrow \pm \infty. 
\end{equation*}
By the compact embedding of $H^{1/2}(\partial D)$ into $L^2(\partial D)$ we have that $v_\infty = 0$ and $w_\infty = 0$ on $\partial D$ since $(w_\infty , v_\infty) \in X(D)$.

Unlike the previous case, due to the normalization in $X(D)$, we can not conclude that the eigenfunctions are nontrivial yet. To prove this, we need to make an extra assumption on the boundary values for the eigenfunctions. In our investigation, we are unable to avoid this technicality just as in \cite{ayala2022analysis}. With this in mind, we now prove a lemma to show that the corresponding limits $(w_\infty , v_\infty)$ are nontrivial.

\begin{lemma}\label{etatoinftyH2:bounded}
If $\partial D$ is of class $\mathscr{C}^2$ or is polygonal with no reentrant corners, $A \in \mathscr{C}^1(\overline{D},\R^{d \times d})$, $v_\eta$ on $\partial D$ is bounded in  $H^{3/2}(\partial D)$ and that $k_\eta$ is bounded,  then $v_\eta$ and $w_\eta$ are bounded in $H^2(D)$.
\end{lemma}
\begin{proof}
Recall that the eigenfunctions $v_\eta$ and $w_\eta$ satisfy
\begin{equation*}
 \Delta v_\eta + k_\eta^2v_\eta = 0 \quad \text{ or }\quad \nabla \cdot A\nabla w_\eta+k_\eta^2nw_\eta =0 \quad \text{in $D$}.
\end{equation*}
Since $v_\eta , w_\eta \in H^1(D) \subset L^2(D)$ and $n \in L^\infty(D)$, we have
\begin{equation*}
\Delta v_\eta = -k_\eta^2v_\eta \in L^2(D) \quad \text{ and }\quad    \nabla \cdot A\nabla w_\eta = -k_\eta^2nw_\eta \in L^2(D) .
\end{equation*}
Due to the regularity of the boundary and coefficient matrix, standard elliptic regularity \cite{salsa} gives that $v_\eta,w_\eta \in H^2(D)$ such that
\begin{equation*}
    ||v_\eta||_{H^2(D)} \leq C_1 \Big[  ||v_\eta||_{L^2(D)} + k_\eta^2||v_\eta||_{L^2(D)} + ||v_\eta||_{H^{3/2}(\partial D)} \Big],
\end{equation*}
\begin{equation*}
    ||w_\eta||_{H^2(D)} \leq C_2 \Big[  ||w_\eta||_{L^2(D)} + k_\eta^2n_{\text{max}}||w_\eta||_{L^2(D)} + ||v_\eta||_{H^{3/2}(\partial D)} \Big],
\end{equation*}
where we have used the fact that $w_\eta=v_\eta$ on $\partial D$. Note that the constants $C_j>0$ for $j=1,2$ are independent of $\eta$. Since $v_\eta$ and $w_\eta$ are bounded in $H^1(D)$ and  $k_\eta$ is bounded, $||v_\eta||_{H^2(D)}$ and  $||w_\eta||_{H^2(D)}$ are bounded.
\end{proof}

The compact embedding of $H^2(D)$  into $H^1(D)$ implies that 
$$||v_\infty||^2_{H^1(D)} + ||w_\infty||^2_{H^1(D)} = 1.$$
Hence the limiting functions are nontrivial. Also, note that Lemma \ref{etatoinftyH2:bounded} still holds provided that $\partial_\nu v_\eta$ on $\partial D$ is bounded in $H^{1/2}(\partial D)$. Put the results together from this subsection, we have proved the following theorem. 

\begin{theorem}\label{tev:etatoinfty}
If $\partial D$ is class $\mathscr{C}^2$ or is polygonal with no reentrant corners as well as $A \in \mathscr{C}^1(\overline{D},\R^{d \times d})$ with $v_\eta$ on $\partial D$ is bounded in  $H^{3/2}(\partial D)$ and that $k_\eta$ is bounded then for $k_\eta \to k_\infty$ as $\eta \to \pm \infty$ we have that $(w_\eta, v_\eta) \to (w_\infty, v_\infty)$ as $\eta \to \pm \infty$ in $X(D)$ such that 
 $$ \Delta v_\infty + k_\infty^2v_\infty = 0 \quad \text{ and }\quad \nabla \cdot A\nabla w_\infty+k_\infty^2nw_\infty =0 \quad \text{in $D$} $$
where $(w_\infty, v_\infty) \in H^1_0(D) \times H^1_0(D)$ nontrivial.
\end{theorem}

The above theorem states that the TEVs limit points as $\eta \to \pm \infty$ are either the Dirichlet eigenvalues for $-\Delta$ or $-n^{-1}\nabla \cdot A\nabla$. The analysis in this section works for any TEV with a limit point as $\eta \to \pm \infty$. Therefore, we conclude that the complex TEVs (provided they exist) limit points  as $\eta \to \pm \infty$ all lie on the real line. To summarize, we have shown the discreteness, existence, and dependence on the parameters for TEVs.

\section{Numerical Validation} \label{sec-numerical}
In this section, we provide numerical examples that demonstrate the theoretical results in the previous sections. We shall focus on the dependence of the TEVs with respect to the coefficients $A,n,$ and $\eta$. From Section \ref{TEVmono}, the first TEV is monotone with respect to the coefficients. Also, in sections \ref{sec-etato0} and \ref{sec-etatoinfty} we established the limiting behavior of the TEVs as $\eta \longrightarrow 0$ and $\eta \longrightarrow \pm \infty$. The following numerical examples demonstrate the theory. 

\begin{figure}[!h]
\centering 
\includegraphics[scale=0.13]{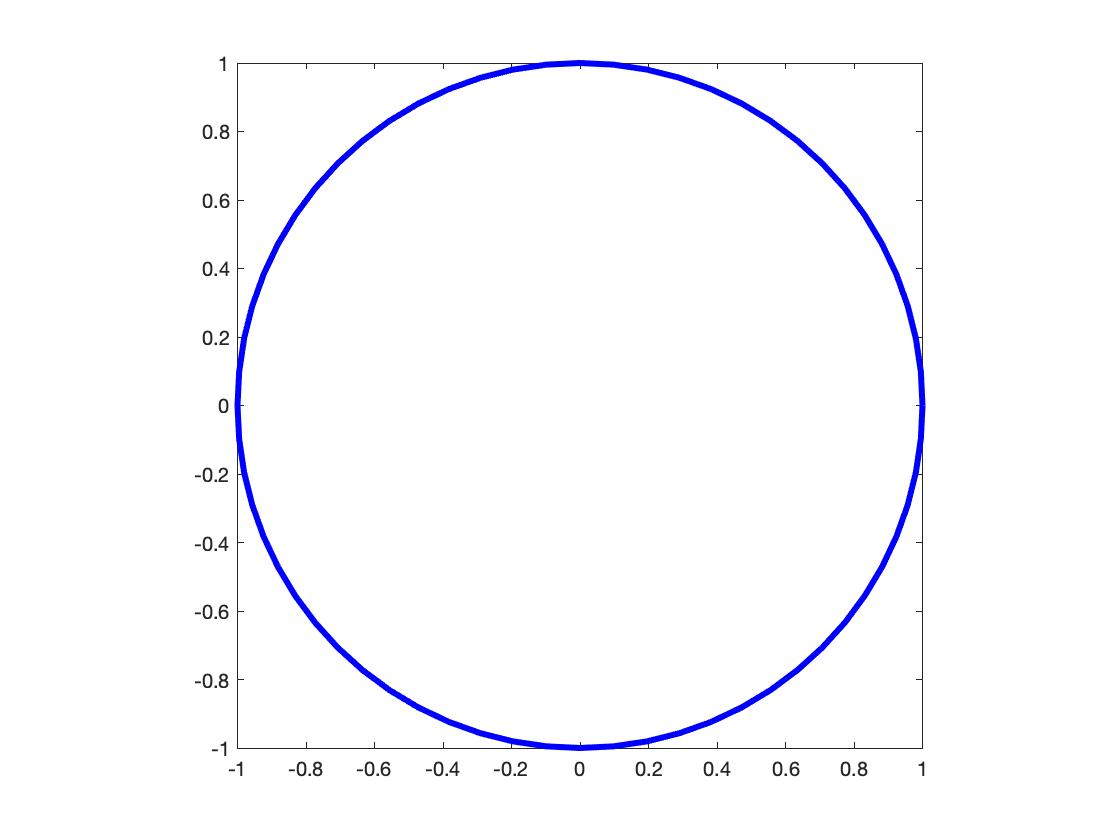} \hspace{-0.2in} \includegraphics[scale=0.13]{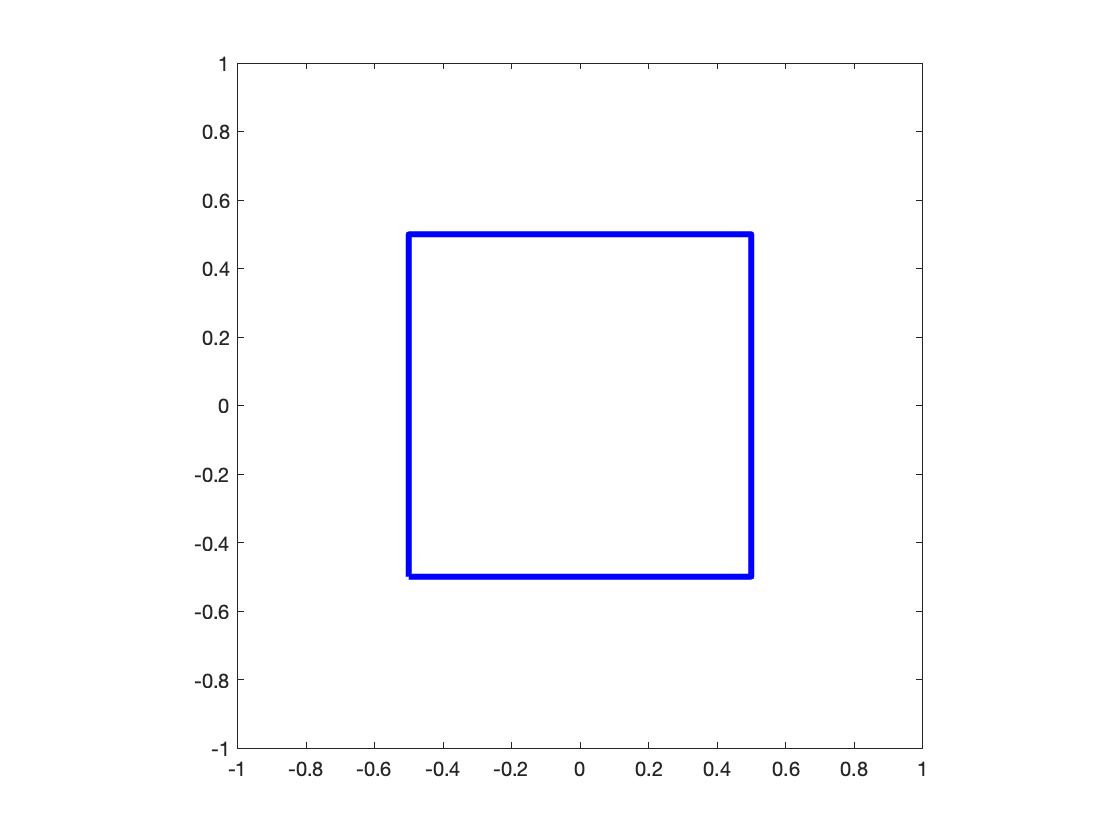}  \hspace{-0.2in} \includegraphics[scale=0.13]{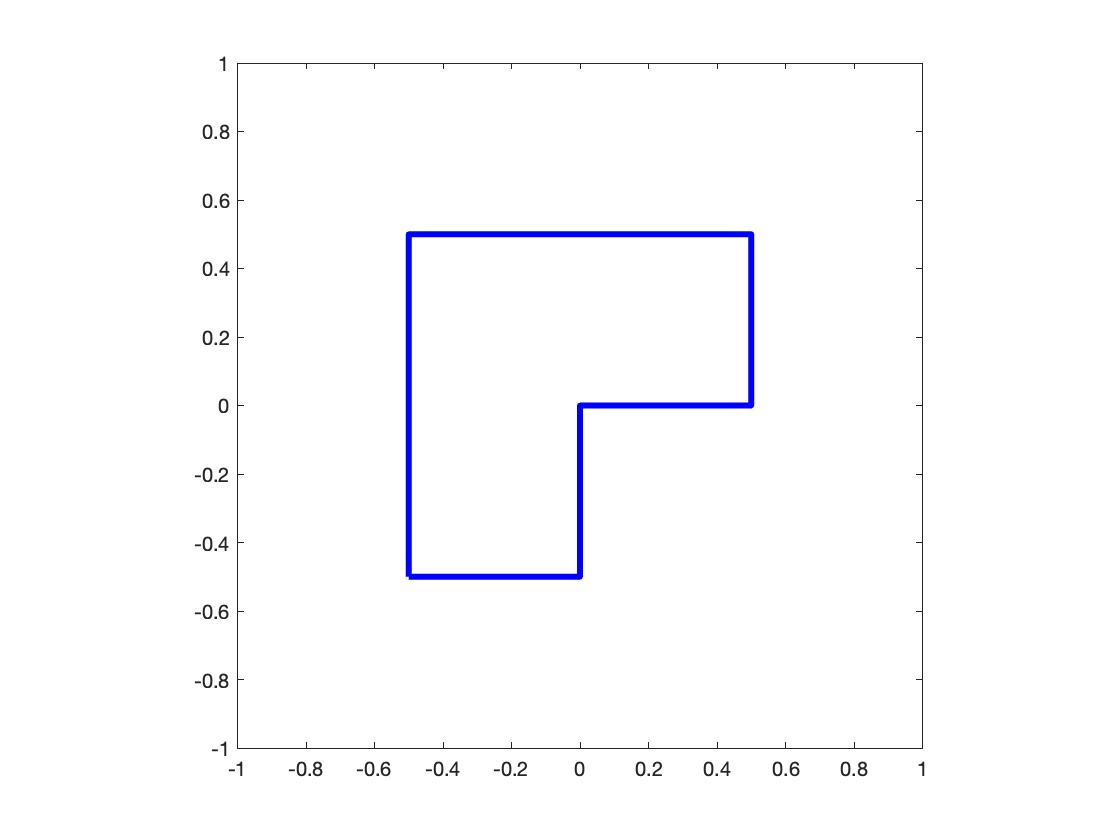}
\caption{ The unit circle, unit square and L--Shaped domains}
\label{TEVdomains}
\end{figure}

We consider three domains in $\R^2$: the unit disk, the unit square, and an L--shaped domain (see Figure \ref{TEVdomains}). For simplicity, we shall assume that $A= a I$ where $a$ is constant and $I$ is the $2 \times 2$ identity matrix.

\subsection{Unit Disk} \label{sec-numericalunitdisk}
For the unit disk, by separation of variables, $k$ is a TEV if and only if $d_m(k) = 0$ for some $m \in \N$, where
\begin{equation}
    d_m(k) := \text{det} 
    \begin{pmatrix}
        J_m\Big(\sqrt{\frac{n}{a}}k \Big) & -J_m(k) \\
        k\sqrt{na}J'_m\Big(\sqrt{\frac{n}{a}}k \Big) & -\Big[ kJ'_m(k) +\eta J_m(k) \Big]
    \end{pmatrix}. 
    \label{TEVdetfunc}
\end{equation}
where $n$ and $\eta$ are constant and $J_m$ denotes the Bessel function of the first kind of order $m$. 

We consider the case when $\eta \to 0$. By theorem \ref{TEVexist} there exist infinitely many real eigenvalues. Theorem \ref{thm-etato0} gives that the limit point as $\eta \to 0$ is the TEV with $\eta=0$, i.e. the TEVs are continuous at $\eta=0$. To verify the convergence we compute the first root of $d_0(k)$ given by \eqref{TEVdetfunc}. We remark that the first root of $d_0(k)$ in general is be the first TEV. The estimated rate of convergence (EOC) is defined as
$$\text{EOC} = \frac{1}{\log(2)} \log \Bigg( \frac{|k_{\eta_p}-k_0|}{|k_{\eta _{p+1}}-k_0|} \Bigg) \quad \text{with  $\eta_p = \pm \frac{1}{2^p}$ for $p=0,1,2,3,...$}$$
The numerical results are given in Table \ref{EOC:1}. 
\begin{table}[!ht]
\begin{center}
\begin{tabular}{c|c|c||c|c|c }
        $\eta$ & $k_{\eta}$ & EOC & $\eta$ & $k_\eta$ & EOC\\
        \hline
        1 & 1.6010 & N/A & -1 & 5.8032 & N/A \\
        1/2 & 1.7353 & 1.7003 &  -1/2 & 5.8143 & 0.9223 \\
        1/4 & 1.7697 & 1.2386 &-1/4 & 5.8203 & 0.9542 \\
        1/8 & 1.7832 & 1.1004 & -1/8 & 5.8234 & 0.9556 \\
        1/16 & 1.7893 & 1.0498 &-1/16 & 5.8250 & 0.9569 \\
        1/32 & 1.7922 & 1.0255 & -1/32 & 5.8258 & 0.9175 \\
        1/64 & 1.7936 & 1 & -1/64 & 5.8262 & 0.8480 \\
        1/128 & 1.7943 & 1 & -1/128 & 5.8264 & 0.7370 \\
        1/256 & 1.7946 & 1 &-1/256 & 5.8265 & 0.5850 \\
        1/512 & 1.7948 & 1 & -1/512 & 5.8266 & 1 \\
        \hline
    \end{tabular}
\end{center}
\caption{Convergence with respect to $\eta \to 0$ for the unit disk ($A=0.4 I$ and $n=3$ for $\eta>0$; $A=3 I$ and $n=0.7$ for $\eta<0$).}\label{EOC:1}
\end{table}

Note that the limit for $A=0.4 I$ and $n=3$ is $k_0=1.7950$ and the limit for $A=3 I$ and $n=0.7$ is $k_0=5.8267$. The values $k_0$'s are given by the first root of $d_0(k; \eta = 0)$. The numerical results in Table \ref{EOC:1} verify the convergence in the previous sections and indicate that the EOC is of first order, i.e., $|k_\eta - k_0| = \mathcal{O}(|\eta|)$ as $\eta \to 0$.

For $\eta \to \pm \infty$, from Theorem \ref{tev:etatoinfty} the limiting value $k_\infty$ is such that either 
$$ \Delta v_\infty + k_\infty^2v_\infty = 0 \quad \text{ or }\quad \nabla \cdot A\nabla w_\infty+k_\infty^2nw_\infty =0 \quad \text{in $D$},$$
where $v_\infty=w_\infty =0$ on $\partial D$. The limiting values are given by 
$$  k_\infty = \sqrt{\lambda_j(D)} \quad \text{ or }\quad k_\infty = \sqrt{\lambda_j(D) \frac{a}{n}} \quad \text{ for some $j \in \N$},$$
where $\lambda_j(D)$ is a Dirichlet eigenvalues of the negative Laplacian in $D$. For the unit disk, $ \sqrt{\lambda_1(D)} =2.4048$. For $a=0.4 $ and $n=3$,
$k_\infty =2.4048$ or $k_\infty = 0.8781$ are possible limit points. 
The EOC is defined by
$$\text{EOC} = \frac{1}{\log(2)} \log \Bigg( \frac{|k_{\eta_p}-k_\infty|}{|k_{\eta _{p+1}}-k_\infty|} \Bigg) \quad \text{ with  $\eta_p = \pm {2^p}$ for $p=0,1,2,3,...$}.$$

\begin{table}[!ht]
\begin{center}
\begin{tabular}{c|c|c||c|c|c }
        $\eta$ & $k_{\eta}$ & EOC & $\eta$ & $k_\eta$ & EOC\\
        \hline
        1 & 1.6010 & N/A &  -1 & 5.8032 & N/A \\
        2 & 1.1415 & 1.4565 &  -2 & 5.7839 & 0.00822 \\
        4 & 0.9828 & 1.3310 &-4 & 5.7537 & 0.01295 \\
        8 & 0.9258 & 1.1342 &  -8 & 5.7124 & 0.0179 \\
        16 & 0.9010 & 1.0586 &-16 & 2.5944 & 4.1248 \\
        32 & 0.8893 & 1.0318  & -32 & 2.4881 & 1.1866 \\
        64 & 0.8837 & 1 &  -64 & 2.4443 & 1.0765 \\
        128 & 0.8809 & 1 & -128 & 2.4241 & 1.0333 \\
        256 & 0.8795 & 1 &-256 & 2.4143 & 1.0266 \\
        512 & 0.8788 & 1 & -512 & 2.4095 & 1.0153 \\
        \hline
    \end{tabular}
\end{center}
\caption{Convergence with respect to $\eta \to \pm \infty$ for the unit disk. ($A=0.4 I$ and $n=3$ for $\eta>0$ and $A=3 I$ and $n=0.7$ for $\eta<0$).}\label{EOC:2}
\end{table}

The results are shown in Table \ref{EOC:2}, which verifies the convergence in the previous sections. Similarly, it seems that EOC is first order i.e. $|k_\eta - k_\infty| = \mathcal{O}\left(1/|\eta|\right)$ as $\eta \to \pm \infty$. Note that by Table \ref{EOC:2} we see that in the limit $k_\eta$ can tend to either a Dirichlet eigenvalues of the negative Laplacian or a Dirichlet eigenvalues for $-n^{-1} \nabla \cdot A \nabla$ in $D$.

\subsection{Validation for the Unit Square and L--Shaped Domain} \label{sec-numericalother}
In this section, we consider two Lipschitz domains, for which a continuous finite element method is used  \cite{ColtonMonkSun2010IP,JiSun2013, Sun2011SIAMNA}. 
The two domains are respectively given by
$$(-1/2, 1/2) \times (-1/2, 1/2) \quad \text{and} \quad (-1/2, 1/2) \times (-1/2, 1/2) \setminus [0,1/2] \times [0,1/2].$$
Again, for simplicity, $A= a I$ where $a$ is constant and $n$ and $\eta$ are constant.

\begin{table}[!ht]
\begin{center}
\begin{tabular}{c|c||c|c }
        $a$ & $k_{1}(a) $ & $a$ & $k_{1}(a) $ \\
        \hline
        $1.5$&$18.3370$ &$0.3$&$3.3113$ \\
        $2$& $14.1549$ & $0.4$& $4.2377$ \\
        $2.5$ &$11.8953$ & $0.5$ &$5.4361$ \\
        $3$ & $10.7253$ & $0.6$ & $6.5176$ \\
        $3.5$& $10.2221$ & $0.7$& $7.6965$ \\
        $4$&$9.9292$ & $0.8$&$9.2953$ \\
        \hline
    \end{tabular}
\end{center}
\caption{Monotonicity with respect to $A=aI$ for the unit square. Here, we take $n=0.75$ and $\eta = -2$ for $a>1$ where as $n=2$ and $\eta=2$ for $a<1$.}\label{mono:1}
\end{table}
In Table \ref{mono:1} we see the monotonicity for the unit square with respect to $A=aI$. 
The monotonicity with respect to $n$ for the L-shaped domain can be found in Table \ref{mono:2}. For both cases Theorems \ref{thm-mono1} and \ref{thm-mono2} are validated. 
\begin{table}[!ht]
\begin{center}
\begin{tabular}{c|r||c|r }
        $n$ & $k_{1}(n)$ & $n$ & $k_1(n)$ \\
        \hline
        $1.5$&$12.9868$ &$0.3$&$10.5159$ \\
        $2$& $9.0795$ & $0.4$& $11.0014$ \\
        $2.5$ &$7.2272$ & $0.5$ & $11.5661$ \\
        $3$ & $6.0632$ & $0.6$ & $12.2386$ \\
        $3.5$& $5.3701$ & $0.7$& $12.9390$ \\
        $4$&$4.8812$ & $0.8$&$13.5370$ \\
        \hline
    \end{tabular}
\end{center}
\caption{Monotonicity with respect to $n$ for the L--Shaped domain. Here, we take $A=0.7 I$ and $\eta=2$ for $n>1$  where as $A=3 I$ and $\eta = -2$ for $n<1$.} \label{mono:2}
\end{table}

Lastly, consider the case when $\eta \to \pm \infty$ for the L-shaped domain. Note that this domain is not covered by Theorem \ref{tev:etatoinfty} due to the reentrant corner. The calculations here are to see if the smoothness of the boundary in Theorem \ref{tev:etatoinfty} is an essential assumption or technical one that can possibly be removed. The results are shown in Table \ref{EOC:3}.
\begin{table}[!ht]
\begin{center}
\begin{tabular}{c|r||c|r }
        $\eta$ & $k_{\eta}$& $\eta$ & $k_\eta$\\
        \hline
        $1$& $3.9995$  &  $-1$& $13.0503$  \\
        $2$ &$3.8281$ & $-2$ &$12.9390$ \\
        $4$ & $3.4828$ & $-4$ & $12.6623$  \\
        $8$& $2.7805$ &  $-8$& $11.7693$   \\
        $16$&$2.4898$  &$-16$&$10.1467$  \\
        $32$&$2.3732$  & $-32$&$7.2242$  \\
        $64$&$2.3196$ &  $-64$&$6.6044$  \\
        $128$&$2.2936$  & $-128$&$6.3938$  \\
        $256$&$2.2807$  &$-256$&$6.3000$  \\
        $512$&$2.2742$   & $-512$&$6.2549$  \\
        \hline
    \end{tabular}
\end{center}
\caption{Convergence with respect to $\eta \to \pm \infty$ for the L-shaped domain ($A=0.4 I$ and $n=3$ for $\eta>0$ and $A=3 I$ and $n=0.7$ for $\eta<0$).}\label{EOC:3}
\end{table}

From Table \ref{EOC:3}, the first TEV does indeed seem to have a limit as $\eta  \to \pm \infty$. The limiting value is given by  
$k_\infty = \sqrt{\lambda_1(D)}$ or $k_\infty = \sqrt{\lambda_1(D) \frac{a}{n}}$ where $\lambda_1(D) \approx 38.5674$.
This is consistent with $k_\infty \approx 6.2103$ or $2.2677$ (i.e. for the case $a= 0.4$ and $n=3$). Therefore, we conjecture that Theorem \ref{tev:etatoinfty} is valid for polygonal domains with reentrant corners. Also, with this we see the `bifurcation' numerically for the limiting case as $\eta \to \pm \infty$ since we have seen for multiple examples that both possible limiting values are valid. 


\section{Conclusions}
We study the scattering problem for an anisotropic material with a conductive boundary. We prove that there exists infinitely transmission eigenvalues and that they are at most a discrete set. We then study the dependence of the transmission eigenvalues on the physical parameters and proved that the first transmission eigenvalue is monotone with respect to each of these parameters. We also prove that the transmission eigenvalues converge as the conductive boundary parameter $\eta$ goes to $0$ or $\infty$ in magnitude. Lastly, we provide some numerical examples for three different shapes, validating our theoretical results for monotonicity and convergence. For the L-shaped domain, the convergence result as $\eta \longrightarrow \pm \infty$ is consistent with the theory, though it isn't covered under our assumptions.
\\

\noindent{\bf Acknowledgments:} The research of I. Harris and V. Hughes is partially supported by the NSF DMS Grant 2107891.\\



\begin{thebibliography}{99}
\bibitem{GLSM} 
\newblock L. Audibert and H. Haddar, 
\newblock { A generalized formulation of the linear sampling method with exact characterization of targets in terms of far-field measurements}, 
\newblock {\it Inverse Problems} {\bf 30}  035011 (2014).


\bibitem{glsm-cracks} 
\newblock L. Audibert, L. Chesnel, H. Haddar, and K. Napal,
\newblock Qualitative indicator functions for imaging crack networks using acoustic waves
\newblock {\it SIAM J. Sci. Comput.,} {\bf 43(2)}  B271--B297 (2021).


\bibitem{corner-nonscatt}
\newblock E. Blasten, L. Paivarinta, and J. Sylvester, 
\newblock Corners always scatter, 
\newblock {\it Comm. Math. Phys.}, {\bf 331(2)}, (2014), 725--753.


\bibitem{te-cbc}
\newblock O. Bondarenko, I. Harris, and A. Kleefeld, 
\newblock The interior transmission eigenvalue problem for an inhomogeneous media with a conductive boundary, 
\newblock {\it Applicable Analysis}, {\bf 96(1)}, (2017), 2--22.


\bibitem{fmconductbc}
\newblock O. Bondarenko and X. Liu,
\newblock The factorization method for inverse obstacle scattering with conductive boundary condition,
\newblock {\it Inverse Problems}, {\bf 29} (2013), 095021.


\bibitem{fm-waveguide}
\newblock  L. Borcea and S. Meng, 
\newblock Factorization method versus migration imaging in a waveguide, 
\newblock {\it  Inverse Problems}, {\bf 35}, (2019), 124006.


\bibitem{cavities} 
\newblock  F. Cakoni, D. Colton, and H. Haddar,
\newblock The interior transmission problem for regions with cavities,
\newblock {\it  SIAM J. Math. Anal.}, {\bf 42},  (2010), 145--162.


\bibitem{far field data} 
\newblock F. Cakoni, D. Colton, and H. Haddar,
\newblock On the determination of Dirichlet or transmission  eigenvalues from far field data, 
\newblock {\it  C. R. Acad. Sci. Paris, Ser. I}, {\bf 348},  (2010), 379--383.


\bibitem{CCH-book} 
\newblock F. Cakoni, D. Colton, and H. Haddar,
\newblock {\it``Inverse Scattering Theory and Transmission Eigenvalues''},
\newblock {CBMS Series, SIAM 88}, {Philadelphia},  (2016).


\bibitem{cakoni2014homogenization} 
\newblock F. Cakoni, H. Haddar, and I. Harris,
\newblock Homogenization approach for the transmission eigenvalue problem for periodic media and application to the inverse problem,
\newblock {\it  Inverse Problems and Imaging}, {\bf 9(4)} (2015), 1025--1049 


\bibitem{On-the-interior-TE} 
\newblock F. Cakoni and A. Kirsch,
\newblock On the interior transmission eigenvalue problem
\newblock  { \it Int. Jour. Comp. Sci. Math.}, {\bf 3}, (2010), 142--167.


\bibitem{fem-te}
\newblock F. Cakoni, P. Monk, and J. Sun, 
\newblock Error analysis of the finite element approximation of transmission eigenvalues, 
\newblock {\it Comput. Methods Appl. Math.}, {\bf 14}, (2014) 419--427. 


\bibitem{aniso-nonscatt} 
\newblock  F. Cakoni, M. Vogelius, and J. Xiao,
\newblock On the Regularity of Non-scattering Anisotropic Inhomogeneities,
\newblock {\it  Archive for Rational Mechanics and Analysis volume}, {\bf 247} 31,  (2023).


\bibitem{ayala2022analysis} 
\newblock R. Ceja Ayala, I. Harris, A. Kleefeld, and N. Pallikarakis, 
\newblock Analysis of the transmission eigenvalue problem with two conductivity parameters,
\newblock {\it  Applicable Analysis}, (2022).


\bibitem{ColtonMonkSun2010IP} 
\newblock D. Colton, P. Monk, J. Sun, 
\newblock { Analytical and computational methods for transmission eigenvalues.} 
\newblock {\it Inverse Problems} {\bf 26} (2010) 045011.


\bibitem{te-geo-paper1}
\newblock H. Diao, X. Cao, and H. Liu, 
\newblock On the geometric structures of transmission eigenfunctions with a conductive boundary condition and applications, 
\newblock {\it Com. in Partial Differential Equation}, {\bf 46(4)}, (2021),  630--679. 

\bibitem{GP}
\newblock D. Gintides and N. Pallikarakis,
\newblock A computational method for the inverse transmission eigenvalue problem, 
\newblock {\it Inverse Problems}, {\bf 29}, (2013), 104010.

\bibitem{regfm2}
\newblock I. Harris, 
\newblock Regularized factorization method for a perturbed positive compact operator applied to inverse scattering,
\newblock {\it  Inverse Problems}, {\bf 39}, (2023),  115007.

\bibitem{two-eig-cbc}
\newblock I. Harris, 
\newblock Analysis of two transmission eigenvalue problems with a coated boundary condition,
\newblock {\it  Applicable Analysis}, {\bf 100(9)}, (2021),  1996--2019.

\bibitem{te-cbc2}
\newblock I. Harris and A. Kleefeld, 
\newblock The inverse scattering problem for a conductive boundary condition and transmission eigenvalues, 
\newblock {\it Applicable Analysis}, {\bf 99(3)}, (2020), 508--529.

\bibitem{te-cbc3}
\newblock  I. Harris and A. Kleefeld,
\newblock Analysis and computation of the transmission eigenvalues with a conductive boundary condition, 
\newblock {\it Applicable Analysis}, {\bf 101(6)}, (2022), 1880--1895.

\bibitem{periodictevs}
\newblock  I. Harris, D.-L. Nguyen, J. Sands, and T. Truong,
\newblock  On the inverse scattering from anisotropic periodic layers and transmission eigenvalues. 
\newblock  {\it Applicable Analysis} {\bf 101(8)}, (2022), 3065--3081.


\bibitem{electro-cbc}
\newblock Y. Hao, 
\newblock Electromagnetic interior transmission eigenvalue problem for an inhomogeneous medium with a conductive boundary,
\newblock {\it  Com. on Pure $\&$ Applied Analysis}, {\bf 19(3)} (2020), 1387--1397.

\bibitem{BEM-FEMtransproblem}
\newblock G. Hsiao, F. Liu, J. Sun and L. Xu,
\newblock A coupled BEM and FEM for the interior transmission problem in acoustics
\newblock {\it  Journal of Computational and Applied Mathematics}, {\bf 235} (2011), 5213--5221.


\bibitem{JiSun2013} 
\newblock X. Ji and J. Sun, 
\newblock { A multi-level method for transmission eigenvalues of anisotropic media.} 
\newblock {\it J. Comput. Phys.} {\bf 255} (2013), 422--435.


\bibitem{mfs-te}
\newblock A. Kleefeld and L. Pieronek,
\newblock The method of fundamental solutions for computing acoustic interior transmission eigenvalues,
\newblock {\it Inverse Problems}, {\bf 34}, (2018), 035007.

\bibitem{mfs-anisote}
\newblock A. Kleefeld and L. Pieronek,
\newblock Computing interior transmission eigenvalues for homogeneous and anisotropic media,
\newblock {\it Inverse Problems}, {\bf 34}, (2018),105007.


\bibitem{kirschbook}
\newblock A. Kirsch A and N. Grinberg, 
\newblock {\it ``The Factorization Method for Inverse Problems''.}
\newblock 1st edition Oxford University Press, Oxford 2008.

\bibitem{armin}
\newblock A. Kirsch and A. Lechleiter,
\newblock The inside-outside duality for scattering problems by inhomogeneous media,
\newblock Inverse Problems {\bf 29} (2013), 104011.


\bibitem{func-analysis}
\newblock E. Kreyszig, 
\newblock {\it ``Introductory Functional Analysis with Applications''},
\newblock Wiley Classics Library, 1989.


\bibitem{LKtelimit}
\newblock L. Pieronek and A. Kleefeld,
\newblock On trajectories of complex-valued interior transmission eigenvalues,
\newblock {\it Inverse Problems and Imaging}, {DOI:}{10.3934/ipi.2023041}, (2023).


\bibitem{salo-nonscatt}
\newblock M. Salo and H. Shahgholian,
\newblock Free boundary methods and non-scattering phenomena,
\newblock {\it Res. Math. Sci.}, {\bf 8(4)}, 58 (2023).


\bibitem{salsa} 
\newblock S. Salsa,
\newblock {\it``Partial Differential Equations in Action From Modelling to Theory''},
\newblock Springer Italia, Milano, (2008).


\bibitem{Sun2011SIAMNA} 
\newblock J. Sun, 
\newblock { Iterative methods for transmission eigenvalues.} 
\newblock {\it SIAM J. Numer. Anal.} {\bf 49} (2011) 1860--1874.



\bibitem{eig-FEM-book}
\newblock J. Sun and  A. Zhou,
\newblock {\it ``Finite element methods for eigenvalue problems''}, 
\newblock Chapman and Hall/CRC Publications, Boca Raton, 1st Edition, (2016).



\end{thebibliography}
\end{document}